\documentclass[a4paper]{article}
\usepackage[utf8]{inputenc}
\usepackage{amsthm}
\usepackage{amsmath}
\usepackage{amssymb}
\usepackage{tikz}

\usepackage{caption}
\usepackage{subcaption}
\usepackage{changepage}
\usepackage[labelformat=simple]{subcaption}
\usepackage{xcolor}

\usepackage{hyperref}
\newcommand{\footremember}[2]{%
    \footnote{#2}
    \newcounter{#1}
    \setcounter{#1}{\value{footnote}}%
}
\newcommand{\footrecall}[1]{%
    \footnotemark[\value{#1}]%
}

\usepackage{enumitem}

\usepackage{a4wide}

\usetikzlibrary{calc}

\theoremstyle{definition}
\newtheorem{defin}{Definition}
\newtheorem{remark}[defin]{Remark}
\newtheorem*{const*}{Construction}
\theoremstyle{plain}
\newtheorem{lem}[defin]{Lemma}
\newtheorem{thm}[defin]{Theorem}

\newtheorem{cor}[defin]{Corollary}

\newcommand{\D}{\mathcal{D}}

\title{On the Metric Dimensions for Sets of Vertices \footnote{Parts of this article have been presented in the 9th Slovenian International Conference on Graph Theory (Bled, 2019) and the Bordeaux Graph Workshop (Bordeaux, 2019).}}
\author{%
  Anni Hakanen \footnote{Corresponding author, email: anehak@utu.fi. Research partially funded by the Magnus Ehrnrooth foundation.}  \footremember{TY}{Department of Mathematics and Statistics, University of Turku, Turku, Finland}%
  \and Ville Junnila \footrecall{TY}%
  \and Tero Laihonen \footrecall{TY}%
  \and Mar\'{i}a Luz Puertas \footnote{Partially supported by grants MINECO MTM2015-63791-R and RTI2018-095993-B-I00.} \footremember{alm}{Department of Mathematics, Universidad de Almer\'{i}a, Almer\'{i}a, Spain}%
  }

\begin{document}

\maketitle

\abstract{Resolving sets were originally designed to locate vertices of a graph one at a time.
For the purpose of locating multiple vertices of the graph simultaneously, $\{\ell\}$-resolving sets 
were recently introduced. In this paper, we present new results regarding the $\{\ell\}$-resolving sets of a graph. In addition to proving general results, we consider $\{2\}$-resolving sets in rook's graphs and connect them to block designs. We also introduce the concept of $\ell$-solid-resolving sets, which is a natural generalisation of solid-resolving sets. We prove some general bounds and characterisations for $\ell$-solid-resolving sets and show how $\ell$-solid- and $\{\ell\}$-resolving sets are connected to each other. In the last part of the paper, we focus on the infinite graph family of flower snarks. We consider the $\ell$-solid- and $\{\ell\}$-metric dimensions of flower snarks. 
In two proofs regarding flower snarks, we use a new computer-aided reduction-like approach. }

\section{Introduction}

The graphs we consider are undirected and simple. They are also connected and finite unless otherwise stated. The vertex set of a graph $G$ is denoted by $V(G)$ or simply by $V$ if the graph in question is clear from context. The \emph{distance} between vertices $v$ and $u$, denoted by $d(v,u)$, is the length of any shortest path between $v$ and $u$.

Consider a graph $G$ with vertices $V$. Let $S = \{s_1, \ldots, s_k\} \subseteq V$ be nonempty. The \emph{distance array} of vertex $v\in V$ with respect to the set $S$ is defined as $\D_S (v) = (d(s_1,v),\ldots,d(s_k,v))$. If no two vertices have the same distance array, the set $S$ is called a \emph{resolving set} of $G$. This concept was introduced independently by Slater \cite{S:leavesTree} and Harary and Melter \cite{Harary76}. Resolving sets have applications in robot navigation \cite{Khuller96} and network discovery and verification \cite{Beerliova06}, for example. For some recent developments, see \cite{EstradaLexi16, KangFrac18, KelencEdge18}.

Resolving sets can be used to locate vertices of a graph one at a time. Our research focuses on how we can locate multiple vertices simultaneously. To that end, let us define the \emph{distance array} of a vertex set $X \subseteq V$ with respect to $S = \{s_1, \ldots, s_k\} \subseteq V$ as 
\begin{align*}
\D_S (X) = (d(s_1,X),\ldots, d(s_k,X)),
\end{align*}
where $d(s_i, X) = \min_{x \in X} \{ d(s_i,x) \}$ for all $s_i \in S$. For any singleton set $\{v\} \subseteq V$ we naturally have $\D_S (\{v\}) = \D_S (v)$. The following definition was introduced in \cite{Laihonen16}.

\begin{defin}\label{def:lresolving}
Let $\ell \geq 1$ be an integer. The set $S \subseteq V(G)$  is an $\{\ell\}$\emph{-resolving set} of $G$, if for all distinct nonempty sets $X,Y \subseteq V(G)$ such that $|X| \leq \ell$ and $|Y| \leq \ell$ we have $\D_S (X) \neq \D_S (Y)$.
\end{defin}

When $\ell = 1$, Definition \ref{def:lresolving} is equivalent to the definition of a resolving set.

Consider the graph $H$ illustrated in Figure \ref{fig:intro}. 
The set $R_1 = \{v_2,v_3,v_7\}$ is a $\{1\}$-resolving set of $H$. The vertex $v_6$ and the set $X= \{v_8,v_9\}$ have the same distance array $\D_{R_1} (v_6) = (2,3,1) = \D_{R_1} (X)$ with respect to the set $R_1$. Thus, the set $R_1$ cannot distinguish $X$ from $v_6$. The set $R_2 = \{v_1,v_2,v_3,v_4,v_8,v_9\}$ is a $\{2\}$-resolving set of $H$, and with it we can distinguish $X$ from $v_6$. Indeed, we have $\D_{R_2} (v_6) = (3,2,3,2,2,2)$ and $\D_{R_2} (X) = (3,2,3,2,0,0)$. Moreover, we can uniquely determine the elements of $X$ using $\D_{R_2} (X)$.

We can also distinguish $v_6$ from $X$ with another type of resolving sets introduced in \cite{Solid}. The set $S \subseteq V$ is a \emph{solid-resolving set} of a graph $G$ if for all $v \in V$ and nonempty $X \subseteq V$ we have $\D_S (v) \neq \D_S (X)$. For example, the set $S_1 = \{v_1,v_2,v_3,v_7,v_8\}$ is a solid-resolving set of the graph $H$. Indeed, we have $\D_{S_1} (v_6) = (3,2,3,1,2)$ and $\D_{S_1} (X) = (3,2,3,1,0)$. Solid-resolving sets give unique distance arrays to all vertices. However, some sets of vertices with at least two elements may share the same distance array. Let $Y = \{v_6,v_8\}$. Now $\D_{S_1} (X) = (3,2,3,1,0) = \D_{S_1} (Y)$, and thus the set $S_1$ is not a $\{2\}$-resolving set of $H$. 

The concept of solid-resolving sets can be generalised for larger sets of vertices. Consider again the graph $H$. We want to be able to distinguish sets with up to two vertices as with a $\{2\}$-resolving set, but we want to also distinguish sets with up to two vertices from sets with three or more vertices. In other words, the aim is to locate the elements of sets with up to two vertices and detect if a set contains at least three vertices. Our $\{2\}$-resolving set $R_2$ can do the former but not the latter; the sets $U = \{v_5,v_7\}$ and $W=\{v_5,v_6,v_7\}$ have the same distance array $\D_{R_2} (U) = (2,1,2,1,1,1) = \D_{R_2} (W)$. As a solution to this problem, we now present the following generalisation of solid-resolving sets.

\begin{figure}
\centering
\begin{subfigure}[b]{0.45\linewidth}
 \centering
 \begin{tikzpicture}[scale=0.8]
 	\draw (0,0) -- (1,1) -- (2.5,1) -- (3.5,.5) -- (3.5,-.5) -- (2.5,-1) -- (1,-1) -- (0,0) -- (1,0) -- (1,1);
 	\draw (1,0) -- (1,-1);
 	\draw (2.5,1) -- (2.5,-1);
	\draw \foreach \x in {(1,0),(1,1),(2.5,-1)}{
		\x node[circle, draw, fill=black,
                        inner sep=0pt, minimum width=6pt] {}
	};
	\draw \foreach \x in {(0,0),(1,-1),(2.5,0),(2.5,1),(3.5,.5),(3.5,-.5)}{
		\x node[circle, draw, fill=white,
                        inner sep=0pt, minimum width=6pt] {}
	};
 	\draw {
 		(-.4,0) node[] {$v_1$}
 		(1,1.4) node[] {$v_2$}
 		(1.4,0) node[] {$v_3$}
 		(1,-1.4) node[] {$v_4$}
 		(2.5,1.4) node[] {$v_5$}
 		(2.9,0) node[] {$v_6$}
 		(2.5,-1.4) node[] {$v_7$}
 		(3.9,.5) node[] {$v_8$}
 		(3.9,-.5) node[] {$v_9$}
 	};
 \end{tikzpicture}
 \caption{The set $R_1$.}
\end{subfigure}
\hfill
\begin{subfigure}[b]{0.45\linewidth}
 \centering
 \begin{tikzpicture}[scale=0.8]
 	\draw (0,0) -- (1,1) -- (2.5,1) -- (3.5,.5) -- (3.5,-.5) -- (2.5,-1) -- (1,-1) -- (0,0) -- (1,0) -- (1,1);
 	\draw (1,0) -- (1,-1);
 	\draw (2.5,1) -- (2.5,-1);
	\draw \foreach \x in {(0,0),(1,0),(1,1),(1,-1),(3.5,.5),(3.5,-.5)}{
		\x node[circle, draw, fill=black,
                        inner sep=0pt, minimum width=6pt] {}
	};
	\draw \foreach \x in {(2.5,0),(2.5,1),(2.5,-1)}{
		\x node[circle, draw, fill=white,
                        inner sep=0pt, minimum width=6pt] {}
	};
 	\draw {
 		(-.4,0) node[] {$v_1$}
 		(1,1.4) node[] {$v_2$}
 		(1.4,0) node[] {$v_3$}
 		(1,-1.4) node[] {$v_4$}
 		(2.5,1.4) node[] {$v_5$}
 		(2.9,0) node[] {$v_6$}
 		(2.5,-1.4) node[] {$v_7$}
 		(3.9,.5) node[] {$v_8$}
 		(3.9,-.5) node[] {$v_9$}
 	};
 \end{tikzpicture}
 \caption{The set $R_2$.}
\end{subfigure}
\caption{The graph $H$ with a $\{1\}$-resolving set and a $\{2\}$-resolving set.}\label{fig:intro}
\end{figure}
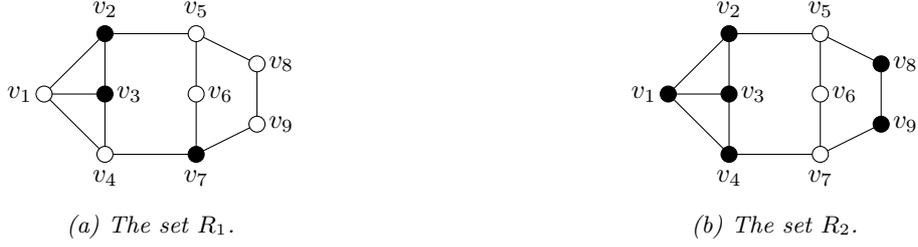

\begin{defin}\label{def:lsolid}
Let $\ell \geq 1$ be an integer. The set $S \subseteq V(G)$  is an $\ell$\emph{-solid-resolving set} of $G$, if for all distinct nonempty sets $X,Y \subseteq V(G)$ such that $|X| \leq \ell$ we have $\D_S (X) \neq \D_S (Y)$.
\end{defin}

When $\ell=1$, the previous definition is exactly the same as the definition of a solid-resolving set in \cite{Solid}. The set $S_2 = \{ v_1,v_2,v_3,v_4,v_6,v_8,v_9 \}$ is a 2-solid-resolving set of $H$. We can distinguish the sets $U$ and $W$ from each other using $S_2$ since $\D_{S_2} (U) = (2,1,2,1,1,1,1)$ and $\D_{S_2} (W) = (2,1,2,1,0,1,1)$.

The difference between Definitions \ref{def:lresolving} and \ref{def:lsolid} is significant but subtle; the set $Y$ can have any cardinality in Definition \ref{def:lsolid}, but in Definition \ref{def:lresolving}, we have the restriction $|Y| \leq \ell$.
If a set $S$ satisfies Definition \ref{def:lsolid} for some $\ell\geq 1$, then $S$ also satisfies Definition \ref{def:lresolving} for the same $\ell$. However, an $\{\ell\}$-resolving set is not necessarily an $\ell$-solid-resolving set (as we saw in the graph $H$).

Since $V(G)$ is an $\ell$-solid-resolving set of $G$ for any $\ell \in \{ 1, \ldots , |V(G)| \}$, it is clear that an $\ell$-solid-resolving set exists for any graph $G$ and any integer $\ell \in \{ 1, \ldots , |V(G)| \}$. Similarly, for any $G$ and $\ell \in \{1, \ldots , |V(G)|\}$ the set $V(G)$ is an $\{\ell\}$-resolving set. Therefore, we focus on determining the minimum cardinality of an $\ell$-solid- or $\{\ell\}$-resolving set of a graph.

The $\{\ell\}$\emph{-metric dimension} of $G$, denoted by $\beta_\ell (G)$, is the minimum cardinality of an $\{\ell\}$-resolving set of $G$. An $\{\ell\}$-resolving set of cardinality $\beta_\ell (G)$ is called an \emph{$\{\ell\}$-metric basis} of $G$.
Similarly, the $\ell$\emph{-solid-metric dimension} of $G$, denoted by $\beta_\ell^s (G)$, is the minimum cardinality of an $\ell$-solid-resolving set of $G$. An $\ell$-solid-resolving set of cardinality $\beta_\ell^s (G)$ is called an \emph{$\ell$-solid-metric basis} of $G$.

We explore the basic properties of $\ell$-solid- and $\{\ell\}$-resolving sets in Section 2. In Section 3, we prove a general lower bound on the $\ell$-solid-metric dimension of a graph and characterise the graphs that attain this bound. In Section 4, we consider Cartesian products of graphs. In particular, we consider the rook's graph $K_m \Box K_n$, and it turns out that the $\{2\}$-metric dimension of a rook's graph is connected to combinatorial designs. Finally, in Section 5, we consider the $\ell$-solid- and $\{\ell\}$-metric dimensions of flower snarks. 
The structure of a flower snark allows us to prove bounds on the 1-solid- and $\{2\}$-metric dimensions by using a new reduction-like approach. We also point out and correct an error in a proof in \cite{ImranFlower} regarding the $\{1\}$-metric dimension of a flower snark.

\section{General Results}

\subsection{The Connection Between $\ell$-Solid- and $\{\ell\}$-Resolving Sets}

The following theorem gives a characterisation for $\ell$-solid-resolving sets. Compared to Definition \ref{def:lsolid}, this characterisation provides a significantly easier way to verify that a set is an $\ell$-solid-resolving set.

\begin{thm}\label{thm:superb}
Let $S \subseteq V$ and $\ell \geq 1$. The set $S$ is an $\ell$-solid-resolving set of $G$ if and only if for all $x \in V$ and nonempty $Y \subseteq V$ such that $x \notin Y$ and $|Y| \leq \ell$ there exists an element $s \in S$ such that
\begin{equation}\label{eq:superb}
d(s,x) < d(s,Y).
\end{equation}
\end{thm}
\begin{proof}
($\Rightarrow$) Assume that $S$ does not satisfy \eqref{eq:superb}. There exists a vertex $x \in V$ and a set $Y \subseteq V$ such that $x \notin Y$, $|Y| \leq \ell$ and $d(s,x) \geq d(s,Y)$ for all $s \in S$. Now $\D_S (Y) = \D_S (Y \cup \{x\})$ and $S$ is not an $\ell$-solid-resolving set of $G$ by Definition \ref{def:lsolid}.

($\Leftarrow$) Assume then that $S$ satisfies \eqref{eq:superb}. Consider nonempty vertex sets $X,Y \subseteq V$ such that $|X| \leq \ell$ and $X \neq Y$. We have the following two cases:
\begin{enumerate}[topsep=2pt,itemsep=1pt]
 \item $Y \not\subset X$:
 Let $y \in Y \setminus X$. Since $S$ satisfies \eqref{eq:superb}, there exists an element $s \in S$ such that $d(s,y) < d(s,X)$. Now we have $\D_S (X) \neq \D_S (Y)$.
 \item $Y \subset X$:
 Since $X \neq Y$, there exists a vertex $x \in X$ such that $x \notin Y$. Furthermore, we have $|Y| < |X| \leq \ell$. According to \eqref{eq:superb}, we have $d(s,x) < d(s,Y)$ for some $s \in S$, and consequently $\D_S (X) \neq \D_S (Y)$.
\end{enumerate}
Thus, the set $S$ is an $\ell$-solid-resolving set of $G$ by Definition \ref{def:lsolid}.
\end{proof}

Theorem \ref{thm:superb} will be very useful throughout the article. This theorem also implies the corresponding result for $\ell = 1$ in \cite[Thm 2.2]{Solid}.
A somewhat similar result holds for $\{\ell\}$-resolving sets as stated in the following lemma. Unlike in Theorem \ref{thm:superb}, we now have only an implication and not an equivalence.

\begin{lem}\label{lem:superbl}
Let $S \subseteq V$ and $\ell \geq 2$. If $S$ is an $\{\ell\}$-resolving set of $G$, then for all $x \in V$ and $Y \subseteq V$ such that $x \notin Y$ and $|Y| \leq \ell -1$ there exists an element $s \in S$ for which we have 
\begin{equation}\label{eq:superbl}
d(s,x) < d(s,Y).
\end{equation}
\end{lem}
\begin{proof}
Assume that $S$ does not satisfy \eqref{eq:superbl}. There exists a vertex $x \in V$ and a set $Y \subseteq V$ such that $x \notin Y$, $|Y| \leq \ell - 1$ and $d(s,x) \geq d(s,Y)$ for all $s \in S$. Now $\D_S (Y) = \D_S (Y \cup \{x\})$ and since $|Y| < |Y \cup \{x\}| \leq \ell$, the set $S$ is not an $\{\ell\}$-resolving set of $G$.
\end{proof}

Now, if $S$ is an $\{\ell + 1\}$-resolving set of $G$ for some $\ell \geq 1$, then according to Lemma \ref{lem:superbl} for all $x \in V$ and $Y \subseteq V$ such that $x \notin Y$ and $|Y| \leq \ell$ there exists an element $s \in S$ such that $d(s,x) < d(s,Y)$. According to Theorem \ref{thm:superb}, the set $S$ is now also an $\ell$-solid-resolving set, and the next result is immediate.

\begin{thm}\label{thm:solidrs}
Let $S \subseteq V$ and $\ell \geq 1$.
\begin{enumerate}[topsep=1pt,itemsep=0pt]
 \item[(i)] If $S$ is an $\ell$-solid-resolving set, then it is an $\{\ell\}$-resolving set of $G$.
 \item[(ii)] If $S$ is an $\{\ell +1\}$-resolving set, then it is an $\ell$-solid-resolving set of $G$.
\end{enumerate}
\end{thm}

If we know that a set $S$ is an $\ell$-solid-resolving set of $G$, then to prove that the set $S$ is an $\{\ell + 1\}$-resolving set of $G$, it is sufficient to check that the distance arrays of vertex sets of cardinality $\ell + 1$ are unique. Indeed, according to Definition \ref{def:lsolid} the distance arrays $\D_S (X)$, where $|X| \leq \ell$, are unique. The only thing we need to do to prove that $S$ satisfies Definition \ref{def:lresolving} is to show that no two vertex sets of cardinality $\ell +1$ have the same distance array with respect to $S$.

\subsection{Forced Vertices}

A vertex $v \in V(G)$ is called a \emph{forced vertex} of an $\{\ell\}$-resolving set (sim. $\ell$-solid-resolving set) of $G$ if it must be included in any $\{\ell\}$-resolving set of $G$. In other words, no subset of $V(G) \setminus \{v\}$ is an $\{\ell\}$-resolving set of $G$. The graph we are considering is often clear from the context, and we may refer to a forced vertex of that graph by saying simply that the vertex is forced for an $\ell$-solid- or $\{\ell\}$-resolving set. The number of forced vertices of an $\ell$-solid- or $\{\ell\}$-resolving set gives us an immediate lower bound on the corresponding metric dimension.

The concept of forced vertices was first introduced in \cite{EllRSandKing}, where the forced vertices of $\{\ell\}$-resolving sets were partially characterised. As was pointed out in \cite{Chartrand00}, the set $V \setminus \{v\}$ is a $\{1\}$-resolving set of a nontrivial connected graph $G$ for all $v \in V$. Thus, no such graph has forced vertices for a $\{1\}$-resolving set. In \cite{Solid}, the forced vertices of 1-solid-resolving sets were fully characterised. In this section, we prove characterisations for $\ell$-solid- and $\{\ell\}$-resolving sets for all $\ell$.

We denote by $N(v)$ the \emph{open neighbourhood} of vertex $v$ which is defined as $N(v) = \{ u\in V \ | \ d(v,u) = 1 \}$. The \emph{closed neighbourhood} of a vertex $v \in V$ is $N[v] = N(v) \cup \{v\}$ and the closed neighbourhood of a vertex set $U$ is $N[U] = \cup_{u \in U} N[u]$.

\begin{thm}\label{thm:solfor}
Let $\ell \geq 1$. A vertex $v \in V$ is a forced vertex of an $\ell$-solid-resolving set of $G$ if and only if there exists a set $U \subseteq V$ such that $v \notin U$, $|U| \leq \ell$ and $N(v) \subseteq N[U]$.
\end{thm}
\begin{proof}
$(\Leftarrow)$ Assume that $v$ and $U$ are as described. The shortest path from any $s \in V \setminus \{v\}$ to $v$ goes through $N(v)$. Since $N(v) \subseteq N[U]$, we have $d(s,v) \geq d(s,U)$ for all $s \in V \setminus \{v\}$. Thus, $\D_S (U) = \D_S (U \cup \{v\})$ for all subsets $S \subseteq V \setminus \{v\}$.

$(\Rightarrow)$ Assume then that $v \in V$ and for all $U \subseteq V$ such that $v \notin U$ and $|U|\leq \ell$ we have $N(v) \nsubseteq N[U]$. Now there exists a vertex $w \in N(v) \setminus N[U]$, and we have $d(w,v) < d(w,U)$. Since $d(x,x) < d(x,Y)$ for any $x \in V$ and $Y \subseteq V \setminus \{x\}$, the set $V \setminus \{v\}$ satisfies \eqref{eq:superb} and is thus an $\ell$-solid-resolving set of $G$, which contradicts the fact that $v$ is forced for an $\ell$-solid-resolving set.
\end{proof}

According to Theorem \ref{thm:solidrs} an $\{\ell\}$-resolving set, where $\ell \geq 2$, is always an $(\ell-1)$-solid-resolving set of the graph in question. Thus, if a vertex is forced for $(\ell-1)$-solid-resolving sets of a graph, then it is also forced for the $\{\ell\}$-resolving sets of the same graph. The following theorem characterises all forced vertices of an $\{\ell\}$-resolving set of a graph, and shows that the forced vertices of $\{\ell\}$-resolving sets are in fact \emph{exactly the same} as those of $(\ell-1)$-solid-resolving sets. 

\begin{thm}\label{thm:resfor}
Let $\ell \geq 2$. A vertex $v \in V$ is a forced vertex of an $\{\ell\}$-resolving set of $G$ if and only if there exists a set $U \subseteq V$ such that $v \notin U$, $|U| \leq \ell -1$ and $N(v) \subseteq N[U]$.
\end{thm}
\begin{proof}
$(\Leftarrow)$ Clear by Theorems \ref{thm:solidrs} and \ref{thm:solfor}.

$(\Rightarrow)$ Assume then that $v \in V$ and that for all $U \subseteq V$ such that $v \notin U$ and $|U|\leq \ell - 1$ we have $N(v) \nsubseteq N[U]$. We will show that the set $S=V \setminus \{v\}$ is an $\{\ell\}$-resolving set of $G$ by showing how to determine the elements of a vertex set $X$ when the distance array $\D_S (X)$ is known. Consider a nonempty set $X \subseteq V$, where $|X| \leq \ell$, and let $\D_S (X)$ be known. We can easily determine the elements of $X' = X \cap S$ by considering the zeros in the distance array $\D_S (X)$. If $|X'| = \ell$, then $X = X'$ and we have uniquely determined all elements of $X$. Otherwise, we still need to determine whether $v$ is in $X$ since it is the only vertex of the graph that is not in $S$. Since $|X'| \leq \ell - 1$ and $v \notin X'$, there exists a vertex $w \in N(v) \setminus N[X']$ according to our assumption. Now, $d(w,v) < d(w,X')$ and $d(w,X) = d(w,v)$ if and only if $v \in X$.
\end{proof}

To illustrate the previous theorems, consider again the graph $H$ in Figure \ref{fig:intro}. Since $N(v_1) = \{ v_2,v_3,v_4 \}$ and $N(v_3) = \{v_1,v_2,v_4\}$, we have $N(v_1) \subseteq N[v_3]$ and $N(v_3) \subseteq N[v_1]$. By Theorems \ref{thm:solfor} and \ref{thm:resfor}, the vertices $v_1$ and $v_3$ are forced vertices of 1-solid- and $\{2\}$-resolving sets of $H$.

Consider then any connected graph $G$. If $\deg (v) \leq \ell$ for some vertex $v$ and integer $\ell \geq 1$, then $N(v) \subseteq N[N(v)]$ and $v$ is forced for $\ell$-solid- and $\{\ell+1\}$-resolving sets of $G$ by Theorems \ref{thm:solfor} and \ref{thm:resfor}. 
In particular, if $G$ is a tree, then a vertex $v$ is forced for $\ell$-solid- and $\{\ell+1\}$-resolving sets if and only if $\deg (v) \leq \ell$.
In \cite{EllRSandKing}, it was shown that the forced vertices of an $\{\ell\}$-resolving set of a tree indeed form an $\{\ell\}$-resolving set, when $\ell \geq 2$. Since any $\{\ell + 1\}$-resolving set is an $\ell$-resolving set and the forced vertices of these two types of resolving sets are exactly the same, the $\ell$-solid-resolving sets of a tree consist of only the corresponding forced vertices. Thus, for any $\ell$ we can construct trees that have nontrivial $\ell$-solid- and $\{\ell\}$-resolving sets.

\section{Bounds and Characterisations}

For the $\{1\}$-metric dimension of a graph there is the obvious lower bound $\beta_1 (G) \geq 1$. This lower bound is attained if and only if $G = P_n$ \cite{Chartrand00, Khuller96}.
In this section, we prove a lower bound on the $\ell$-solid-metric dimension of a graph and characterise the graphs attaining that bound. The lower bound $\beta_1^s (G) \geq 2$ on the 1-solid-metric dimension of a graph was shown in \cite{Solid}. The following theorem generalises this lower bound for $\ell$-solid-metric dimensions where $\ell \geq 2$.

\begin{thm}\label{thm:lslowerbound}
Let $G$ be a graph with $n$ vertices. When $1 \leq \ell \leq n-1$, we have $\beta_\ell^s (G) \geq \ell + 1$.
\end{thm}
\begin{proof}
Let $S \subseteq V$ such that $1 \leq |S| \leq \ell$. Since $\ell \leq n-1$, there exists at least one vertex $v$ which is not in $S$. Now, $\D_S (S) = (0,\ldots,0) = \D_S (S \cup \{v\})$, and $S$ is not an $\ell$-solid-resolving set of $G$ according to Definition \ref{def:lsolid}.
\end{proof}

The following theorem characterises the graphs attaining the bound of Theorem \ref{thm:lslowerbound}.

\begin{thm}
Let $G$ be a connected graph with $n$ vertices and let $2\leq \ell \leq n-1$. We have
\begin{align*}
\beta_{\ell}^s(G)=\ell +1 \text{ if and only if } n=\ell +1 \text{ or } G= K_{1,\ell + 1}.
\end{align*}
\end{thm}

\begin{proof}
If $n=\ell +1$, on the one hand, $\beta_{\ell}^s(G)\leq n=\ell+1$ and on the other hand $\beta_{\ell}^s(G)> \ell$, and thus $\beta_{\ell}^s(G)=\ell+1$. Also, by Theorem 2.9 of \cite{EllRSandKing}, the star $K_{1,\ell+1}$ with $\ell +2 $ vertices satisfies $\beta_{\ell+1}(K_{1,\ell+1})=\ell+1$. Therefore $\ell<\beta_{\ell}^s(K_{1,\ell+1})\leq \beta_{\ell+1}(K_{1,\ell+1})=\ell+1$, and thus $\beta_{\ell}^s(K_{1,\ell+1})=\ell +1$.

Conversely, suppose that $G$ is a connected graph such that $|V|=n\geq \ell +2$ and $\beta_{\ell}^s(G)=\ell +1$ and let $S\subseteq V$ be an $\ell$-solid-resolving set with $\ell +1$ vertices. The following properties hold.

\begin{enumerate}

\item The set $S$ is independent: Suppose to the contrary that there exist $s_1,s_2 \in S$ such that $d(s_1,s_2)=1$. Since $|V|\geq \ell+2$, there exists $u\in V\setminus S$ that satisfies $d(u,s_1)\geq 1=d(s_2,s_1)$. Now, $d(v,u) \geq d(v, S \setminus \{s_1\})$ for all $v \in S$, and since $|S\setminus\{s_1\}|=\ell$, the set $S$ is not an $\ell$-solid-resolving set of $G$ according to \eqref{eq:superb}, when $x=u$ and $Y=S \setminus \{s_1\}$.

\item We have $\deg(s)=1$ for every $s\in S$: Denote $S = \{s_1, \ldots, s_{\ell+1} \}$. Since $G$ is connected and $S$ is independent, each $s_i$ has a neighbour in $V \setminus S$, say $v_i \in N(s_i)$ for $i = 1, \ldots , \ell +1$. Suppose to the contrary that $\deg (s_i) \geq 2$ for some $i$. Assume without loss of generality that $\deg (s_1) \geq 2$. There exists a vertex $v_1' \in N(s_1)$, $v_1' \neq v_1$. Let $A = \{ v_1, \ldots , v_\ell \}$. Since $S$ is an $\ell$-solid-resolving set of $G$, according to Theorem \ref{thm:superb} we must have $d(s_i,v_1') < d(s_i,A)$ for some $i \in \{1, \ldots, \ell+1\}$. However, we have $d(s_i,A)=1$ for all $i \in \{1, \ldots, \ell \}$, and thus $d(s_{\ell +1}, v_1') < d(s_{\ell+1}, A)$. Specifically, we have $d(s_{\ell +1}, v_1') < d(s_{\ell+1}, v_1)$. Similarly, for $v_1$ and $B = \{ v_1', v_2, \ldots, v_\ell \}$ we have $d(s_{\ell+1} , v_1) < d(s_{\ell+1}, B)$, and specifically $d(s_{\ell+1}, v_1) < d(s_{\ell+1}, v_1')$, a contradiction. Thus, $\deg (s) = 1$ for all $s \in S$.
\end{enumerate}

We now consider two cases.

\begin{enumerate}
\item[Case 1:] There exists $u\in V\setminus S$ and two different vertices $s_1,s_2\in S$ such that $d(u,s_1)=d(u,s_2)=1$. If $|V\setminus S|\geq 2$, then let $v\in V\setminus S$ be such that $v\neq u$. Let $X=(S\setminus \{s_1,s_2\})\cup \{u\}$ and $Y=X\cup \{v\}$. We obtain that $\mathcal{D}_S(X)=\mathcal{D}_S(Y)=(1,1,0,\dots, 0)$, a contradiction. This means that, in this case, $V\setminus S=\{u\}$, and since $u$ is not a forced vertex, $\deg(u)\geq \ell +1$, and thus $u$ is a neighbour of every vertex in $S$. Finally, $G=K_{1,\ell+1}$ because $S$ is independent.

\par\medskip

\item[Case 2:] Every vertex in $V\setminus S$ has at most one neighbour in $S$. As seen above, we know that every vertex in $S$ has exactly one neighbour in $V\setminus S$. We denote $S=\{s_1, \dots , s_{\ell +1}\}$ and $A=\{v_1,\dots , v_{\ell +1}\}$ ($|A|=\ell +1$) where $v_i$ is the unique neighbour of $s_i$, for $1\leq i\leq \ell+1$, and note that $\ell +1\geq 3$. The following properties hold.

\begin{enumerate}

\item The set $A$ is independent. Suppose to the contrary that, say, $v_1$ and $v_2$ are neighbours. Thus, $d(v_1,s_2)=2$. Define the sets $X=(S\setminus \{s_1,s_2\})\cup \{v_1\}$ and $Y=X\cup \{ v_3\}$. Clearly $\mathcal{D}_S(X)=\mathcal{D}_S(Y)=(1,2,0,\dots, 0)$, a contradiction.

\item No pair of vertices of $A$ has a common neighbour. Suppose to the contrary (without loss of generality) that there exists $w\in V$ that satisfies $d(v_1,w)=d(v_2,w)=1$. Then $\deg(w)\geq 2$ and $w\notin S$. Moreover $w\notin A$, because $A$ is independent. Let $X=(S\setminus \{s_1,s_2\})\cup \{w\}$ and $Y=X\cup \{ v_3\}$. Then $\mathcal{D}_S(X)=\mathcal{D}_S(Y)=(2,2,0,\dots, 0)$, a contradiction.
\end{enumerate}

Note that every $v_i\in A$ has at least $\ell$ neighbours ($\ell \geq 2$) in $V\setminus S$, say $\{v_{i,j} \ | \ 1\leq j\leq \ell\}$, because it is not forced. The last property gives that $v_{i,j}\neq v_{i',j'}$ for $(i,j)\neq (i',j')$.

Assume, without loss of generality, that $d(v_{1,1}, s_{\ell +1})=\min \{d(v_{1,j}, s_{\ell +1}) \ | \ 1\leq j\leq \ell\}$ 
and let $X=\{ v_{1,1}, v_{2,1}, \dots , v_{\ell, 1}\}$.
Then for all $s_i \in S$, where $i \neq \ell+1$, we have $d(s_i, X) = 2 \leq d(s_i,v_{1,2})$ since $A$ and $S$ are both independent. Furthermore, since $d(v_{1,1}, s_{\ell +1})\leq d(v_{1,2}, s_{\ell +1})$, we have $d(s_{\ell+1},X) \leq d(s_{\ell+1}, v_{1,1}) \leq d(s_{\ell+1},v_{1,2})$. Thus, $d(s_i,v_{1,2}) \geq d(s_i, X)$ for all $s_i \in S$, and $S$ is not an $\ell$-solid-resolving set of $G$ by Theorem \ref{thm:superb}, a contradiction.
\end{enumerate}
\end{proof}

Notice that the number of graphs that attain the lower bound $\beta_\ell^s (G) \geq \ell + 1$ is infinite when $\ell=1$ and finite when $\ell \geq 2$. 
Corresponding results for $\{\ell\}$-resolving sets can be found in \cite{EllRSandKing}.

Let us then consider infinite graphs, that is, graphs with infinitely many vertices. In \cite{Caceres12}, it was shown that an infinite graph may have finite or infinite $\{1\}$-metric dimension. We will show that the $\{\ell\}$-metric dimension, where $\ell \geq 2$, is infinite for any infinite graph. Moreover, the $\ell$-solid-metric dimension of any infinite graph is infinite. To prove these results, we will consider doubly resolving sets.

\begin{defin}[\cite{Caceres07}]
Let $G$ be a graph with $|V(G)|\geq 2$. Two vertices $v, w \in V (G)$ are \emph{doubly resolved} by $x, y \in V (G)$ if
$d(v, x)-d(w, x)\neq d(v, y)-d(w, y)$.
A set of vertices $S \subseteq V (G)$ doubly resolves $G$, and $S$ is a \emph{doubly resolving set}, if
every pair of distinct vertices $v, w \in  V (G)$ is doubly resolved by two vertices in $S$.
\end{defin}

In \cite{Solid}, it was shown that a 1-solid-resolving set of $G$ is a doubly resolving set of $G$. 
According to Theorem \ref{thm:solidrs} any $\{\ell\}$-resolving set, where $\ell \geq 2$, and $\ell$-solid-resolving set is a 1-solid-resolving set. The following result is now immediate.

\begin{cor}\label{cor:doubly}
If $S\subseteq V(G)$ is an $\{\ell\}$-resolving set ($\ell \geq 2$) or an $\ell$-solid-resolving set of $G$ ($\ell \geq 1$), then $S$ is a doubly resolving set of $G$.
\end{cor}

\begin{lem}[\cite{Caceres12}]\label{lem:doublyinf}
If $G$ is an infinite graph, then any doubly resolving set of $G$ is infinite.
\end{lem}

The following corollary is now immediate due to Corollary \ref{cor:doubly} and Lemma \ref{lem:doublyinf}.

\begin{cor}
If $G$ is an infinite graph, then $\beta_\ell (G) = \infty$, when $\ell \geq 2$, and $\beta_\ell^s (G) = \infty$, when $\ell \geq 1$.
\end{cor}

\section{On Cartesian Products of Graphs}

The \emph{Cartesian product} of the graphs $G$ and $H$ is the graph $G \Box H$ with the vertex set $\{av \ | \ a \in V(G), \ v \in V(H)\}$. Distinct vertices $av,bu \in V(G \Box H)$ are adjacent if $a=b$ and $v \in N_H(u)$, or $a \in N_G(b)$ and $v=u$. We have $d_{G \Box H} (av,bu) = d_G (a,b) + d_H (v,u)$. To simplify notations, we may denote $V$ instead of $V(G \Box H)$ and omit the subscript $G \Box H$ from the distance function. The \emph{projection} of $X \subseteq V$ onto $G$ is the set $\{ x_1 \in V(G) \ | \ x_1 x_2 \in X \}$. Similarly, the projection of $X \subseteq V$ onto $H$ is the set $\{ x_2 \in V(H) \ | \ x_1 x_2 \in X \}$.

\begin{thm}
Let $G$ and $H$ be nontrivial connected graphs and $\ell \geq 1$.
\begin{enumerate}
\item If $S$ is an $\ell$-solid-resolving set of $G\Box H$, then the projection of $S$ onto $G$ (respectively onto $H$) is an $\ell$-solid-resolving of $G$ (respectively of $H$).
\item If $T$ is an $\ell$-solid-resolving set of $G$ and $U$ is an $\ell$-solid-resolving set of $H$, then $T\times U$ is an $\ell$-solid-resolving of $G\Box H$.
\item We have $\max \{\beta_{\ell}^s(G), \beta_{\ell}^s(H)\}\leq \beta_{\ell}^s(G\Box H)\leq  \beta_{\ell}^s(G)\cdot \beta_{\ell}^s(H)$.
\end{enumerate}
\end{thm}
\begin{proof}
1. Let $a\in V(G)$ and $Y\subseteq V(G)$, $|Y| \leq \ell$, and let $h_0\in V(H)$ be a fixed vertex. Let $ah_0\in V(G\Box H)$ and $Y_0=Y\times \{h_0\}$. Now $|Y_0| \leq \ell$ and there exists $s=g_s h_s \in S$ such that
    \begin{align*}
    d_G(g_s,a)+d_H(h_s,h_0) & = d(g_s h_s, a h_0)<d(g_s h_s,Y_0)\\
    & = \min\{ d_G(g_s,y)+d_H(h_s,h_0) \ | \ y \in Y\} \\
    & = \min\{ d_G(g_s,y) \ | \ y \in Y\} + d_H(h_s,h_0).
    \end{align*}
Therefore, $d_G(g_s,a)<\min\{ d_G(g_s,y) \ | \ y \in Y\} = d_G(g_s,Y)$, as desired.

2. Let $a b \in V(G\Box H)$ and $Y\subseteq V(G\Box H)$ such that $|Y| \leq \ell$. Then the projections $Y_G$ and $Y_H$ of $Y$ onto $G$ and $H$, respectively, satisfy $|Y_G| , |Y_H| \leq \ell$. Therefore, there exist $t\in T$ and $u\in U$ such that $d_G(t,a) < d_G(t,Y_G)=\min \{ d_G(t,y_g) \ | \ y_g \in Y_G\}$ and  $d_H(u,b) < d_H(u,Y_H)=\min \{ d_H(u,y_h)\ | \ y_h \in Y_H\}$.

Note that $\min \{ d_G(t,y_g) \ | \ y_g \in Y_G\}+\min \{ d_H(u,y_h) \ | \ y_h \in Y_H\} \leq \min\{ d_G(t,\alpha)+d_H(u,\beta) \ | \ \alpha \beta \in Y\} = \min \{ d(t u,\alpha \beta) \ | \ \alpha \beta \in Y\}=d(t u ,Y).$

Finally, $d(tu,ab)=d_G(t,a)+d_H(u,b)<\min \{ d_G(t,y_g) \ | \ y_g \in Y_G\}+\min \{ d_H(u,y_h) \ | \ y_h\in Y_H\}\leq d(tu,Y)$, as desired.

3. The lower bound follows from 1. and the upper bound follows from 2.
\end{proof}

Notice that in 2., it would be sufficient that the set $U$ satisfies the condition \eqref{eq:superb} with equality, that is, for all $x \in V(H)$ and nonempty $Y \subseteq V(H)$ such that $x \notin Y$ and $|Y| \leq \ell$ there exists $u\in U$ such that $d_H (u,x) \leq d_H (u,Y)$.

\begin{thm}
Let $G$ and $H$ be nontrivial connected graphs and $\ell \geq 2$.
\begin{enumerate}
\item If $S$ is an $\{\ell\}$-resolving set of $G\Box H$, then the projection of $S$ onto $G$ (respectively onto $H$) is an $\{\ell\}$-resolving of $G$ (respectively of $H$).
\item If $S$ is an $\{\ell\}$-resolving set of $G$ (respectively of $H$) and $S'$ is an $\ell$-solid-resolving set of $H$ (respectively of $G$), then $S\times S'$ (respectively $S'\times S$) is a $\{\ell\}$-resolving set of $G\Box H$.
\item We have $\max \{\beta_{\ell}(G), \beta_{\ell}(H)\}\leq \beta_{\ell}(G\Box H)\leq \min \{ \beta_{\ell}(G)\cdot \beta_{\ell}^s(H),  \beta_{\ell}^s(G) \cdot \beta_{\ell}(H)\}$.
\end{enumerate}
\end{thm}
\begin{proof}
1. Let $S$ be an $\{\ell\}$-resolving set of $G\Box H$, and $X$ and $Y$ be subsets of $V(G)$ such that $X \neq Y$, $1 \leq |X| \leq \ell$ and $1 \leq |Y| \leq \ell$. Define $X_0 = X \times \{h_0\}$ and $Y_0 = Y \times \{h_0\}$, where $h_0 \in H$. Clearly, we have $X_0 \neq Y_0$, $|X| = |X_0|$ and $|Y| = |Y_0|$. Hence, there exists a vertex $s = g_s h_s \in S$ such that $d(s, X_0) \neq d(s, Y_0)$. Therefore, as $d(s, X_0) = d_G(g_s, X) + d_H(h_s, h_0)$ and $d(s, Y_0) = d_G(g_s, Y) + d_H(h_s, h_0)$, we obtain that $d_G(g_s, X) \neq d_G(g_s, Y)$. Thus, the projection of $S$ onto $G$ is an $\{\ell\}$-resolving set of $G$. Analogously, it can be shown that the projection of $S$ onto $H$ is an $\{\ell\}$-resolving set of $H$.

2. Let $S$ be an $\{\ell\}$-resolving set of $G$ and $S'$ be an $\ell$-solid-resolving set of $H$. Assume that $X, Y \subseteq V(G\Box H)$ are such that $X \neq Y$, $1 \leq |X| \leq \ell$ and $1 \leq |Y| \leq \ell$. Denote $X = \{g_1 h_1, \ldots, g_k h_k\}$ and $Y = \{g'_1 h'_1, \ldots, g'_{k'} h'_{k'} \}$, where $k = |X|$, $k' = |Y|$, $g_i, g'_i \in V(G)$ and $h_i, h'_i \in V(H)$. Further denote $X_G = \{g_1, \ldots, g_k\}$ and $Y_G = \{g'_1, \ldots, g'_{k'}\}$, and $X_H = \{h_1, \ldots, h_k\}$ and $Y_H = \{h'_1, \ldots, h'_{k'}\}$. The proof now divides into the following two cases:
\begin{itemize}
\item Suppose that $X_G \neq Y_G$. Now there exists a vertex $s \in S$ such that $d_G(s, X_G) \neq d_G(s, Y_G)$. Without loss of generality, we may assume that $d_G(s, g_1) = d_G(s, X_G) < d_G(s, Y_G)$. Observe that by the condition~\eqref{eq:superb} there exists $s' \in S'$ such that $d_H(s', h_1) < d_H(s', h)$ for any $h \in Y_H \setminus \{h_1\}$ since $|Y_H \setminus \{h_1\}| \leq \ell$; we agree that if $Y_H \setminus \{h_1\} = \emptyset$, then any $s' \in S'$ meets the required (empty) condition (similar agreement is also made in the case with $X_G = Y_G$). Thus, we have a vertex $s' \in S$ satisfying $d_H(s', h_1) = d_H(s', Y_H)$. Therefore, we obtain that $d(ss', X) \leq d(ss', g_1 h_1) = d_G(s,g_1) + d_H(s',h_1) < d_G(s, Y_G) + d_H(s', Y_H) \leq  d(ss', Y)$.

\item Suppose that $X_G = Y_G$. Since $X \neq Y$,  we have $X \triangle Y = (X \setminus Y) \cup (Y \setminus X) \neq \emptyset$ and, without loss of generality, we may assume that $g_1 h_1 \in X \triangle Y$. By the condition~\eqref{eq:superbl}, there exists $s \in S$ such that $d_G(s, g_1) < d_G(s, g)$ for any $g \in Y_G \setminus \{g_1\}$ since $|Y_G \setminus \{g_1\}| \leq \ell -1$. Analogously, by~\eqref{eq:superb}, there exists $s' \in S'$ such that $d_H(s', h_1) < d_H(s', h')$ for any $h' \in Y_H \setminus \{h_1\}$ since $|Y_H \setminus \{h_1\}| \leq \ell$. 
For any $g'_i h'_i \in Y$ we have $g'_i \neq g_1$ or $h'_i \neq h_1$ since $g_1 h_1 \notin Y$.
Now $d(ss', X) \leq d(ss', g_1 h_1) = d_G(s, g_1) + d_H(s', h_1) < d_G(s, g'_i) + d_H(s', h'_i) = d(ss', g'_i h'_i)$ for any $g'_i h'_i \in Y$. Hence, we have shown that $d(ss', X) < d(ss', Y)$.
\end{itemize}
Thus, $S \times S'$ is an $\{\ell\}$-resolving set of $G\Box H$. The other claim can be proven analogously.

3. The lower bound follows from 1. and the upper bound follows from 2.
\end{proof}

\subsection{The Rook's Graph $K_m \Box K_n$}

The graph $K_m \Box K_n$ can be illustrated as a grid, see Figure \ref{fig:rook}. A \emph{column} of $K_m \Box K_n$ is the set $\{vu \ | \ u \in V(K_n)\}$ for some fixed $v \in V(K_m)$. Similarly, a \emph{row} of $K_m \Box K_n$ is the set $\{vu \ | \ v \in V(K_m)\}$ for some fixed $u \in V(K_n)$. Two vertices are adjacent if and only if they are on the same row or column. Moreover, if two distinct vertices $x$ and $y$ are on different rows and columns, we have $d(x,y) = 2$. 

Consider any $K_m \Box K_n$ where $m,n \geq 2$. Let $x$, $y$ and $z$ be distinct vertices such that $x$ and $y$ are on the same column, and $x$ and $z$ are on the same row. Any neighbour of $x$ is in the closed neighbourhood of either $y$ or $z$. Thus, we have $N(x) \subseteq N[\{y,z\}]$. According to Theorems \ref{thm:solfor} and \ref{thm:resfor}, $x$ is a forced vertex for $\ell$-solid-resolving sets when $\ell \geq 2$ and $\{\ell\}$-resolving sets when $\ell \geq 3$. Consequently, $\beta_\ell^s (K_m \Box K_n) = mn$ for all $\ell \geq 2$ and $\beta_\ell (K_m \Box K_n) = mn$ for all $\ell \geq 3$. 

The 1-solid- and $\{1\}$-metric dimensions of $K_m \Box K_n$ were considered in \cite{Solid} and \cite{Caceres07}, respectively. Thus, the only $\ell$-solid- or $\{\ell\}$-metric dimension of $K_m \Box K_n$ yet to be determined is the $\{2\}$-metric dimension. In what follows, we show a characterisation for the $\{2\}$-resolving sets of $K_m \Box K_n$. As it turns out, this characterisation provides us an exciting connection between combinatorial designs and $\{2\}$-resolving sets of $K_m \Box K_n$.

A \emph{quadruple} of $K_m \Box K_n$ is the set $\{av,au,bv,bu\}$ where $a,b \in V(K_m)$ and $v,u \in V(K_n)$ are distinct. For example, in $K_7 \Box K_7$ illustrated in Figure \ref{fig:rook}, the set $\{v_1u_1 , v_1u_3 , v_4u_1 , v_4u_3 \}$ is a quadruple, and we can see that these four vertices lie on the corners of a rectangle.

\begin{lem}\label{lem:quadruple}
Let $m,n \geq 2$. If the set $S$ is a $\{2\}$-resolving set of $K_m \Box K_n$, then each quadruple contains at least one element of $S$.
\end{lem}
\begin{proof}
Let $Q = \{av,au,bv,bu\} \subseteq V(K_m \Box K_n)$ be a quadruple that does not contain any elements of $S$. Let us denote $X=\{av,bu\}$ and $Y = \{au,bv\}$. Since $N[X] = N[Y]$, we have $d(s,X) = 1$ if and only if $d(s,Y) = 1$ for all $s \in S$. Consequently, $\D_S (X) = \D_S (Y)$ and $S$ is not a $\{2\}$-resolving set of $K_m \Box K_n$, a contradiction.
\end{proof}

In the following theorem, we show that there are two types of $\{2\}$-resolving sets of $K_m \Box K_n$.

\begin{thm}\label{thm:rookconds}
Let $m,n \geq 2$. If the set $S$ is a $\{2\}$-resolving set of $K_m \Box K_n$, then  
\begin{enumerate}
\item the set $\{v\} \cup (V \setminus N(v))$ is a subset of $S$ for some $v \in V$ or
\item each row and column contains at least two elements of $S$ and each quadruple contains at least one element of $S$.
\end{enumerate}
\end{thm}
\begin{proof}
Suppose first that for some $a \in V(K_m)$ the column $C = \{au \ | \ u \in V(K_n)\}$ does not contain elements of $S$. Let $av \in C$ and $bv \in V \setminus C$ for some $v \in V(K_n)$. Since the column $C$ does not contain any elements of $S$, we have $N(av) \cap S \subseteq N[bv]$. Consequently, $d(s,av) \geq d(s,bv)$ for all $s \in S$, and the set $S$ is not a $\{2\}$-resolving set of $K_m \Box K_n$ according to Lemma \ref{lem:superbl}. Thus, if $S$ is a $\{2\}$-resolving set of $K_m \Box K_n$, then each column (and row, by symmetry) contains at least one element of $S$.

Suppose then that $C \cap S = \{au\}$. Let $b \in V(K_m) \setminus \{a\}$ and $t \in V(K_n) \setminus \{u\}$. Consider the sets $X= \{bu,bt\}$ and $Y= \{bu,at\}$.
For some $cv \in S$, we have $d(cv,X) \neq d(cv,Y)$. As $bu$ is in both $X$ and $Y$, we have $d(cv,bt) \neq d(cv,at)$. Since $at$ and $bt$ are on the same row, $cv$ is either on the column $C$ or the column $D = \{bw \ | \ w \in V(K_n)\}$. The only element of $S$ in $C$ is $au$. However, the element $bu$ is in both $X$ and $Y$, and we have $d(au,X)=d(au,Y) = 1$. Thus, $cv$ must be in $D$. The column $D$ contains the element $bu$, and thus $d(cv,X) = d(cv,Y) = 1$ if $cv \neq bt$. Therefore, we have $cv = bt$ and $bt \in S$. Since this holds for all $b \neq a$ and $t \neq u$, we have that $w \in S$ for all $w \in V \setminus N(au)$.

In conclusion, if $S$ is a $\{2\}$-resolving set of $K_m \Box K_n$ and some column (or row) contains only one element of $S$, the set $\{v\} \cup (V \setminus N(v))$ is a subset of $S$ for some $v \in V$. If each row and column contains at least two elements of $S$, each quadruple contains at least one element of $S$ according to Lemma \ref{lem:quadruple}.
\end{proof}

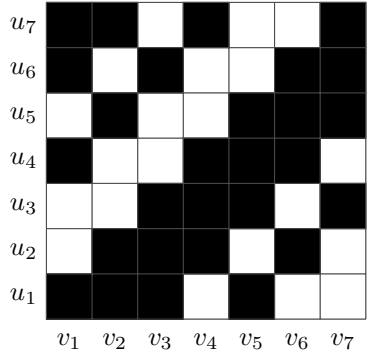
\begin{figure}
\centering
 \begin{tikzpicture}[scale=0.6]
 	\draw[black] (0,0) grid (7,7);
 	\foreach \x in {2,4,5,6} {
 	  \filldraw[fill=black, draw=black] (6,\x) rectangle (7,\x+1);
 	};
 	\foreach \x in {1,3,4,5} {
 	  \filldraw[fill=black, draw=black] (5,\x) rectangle (6,\x+1);
 	};
 	\foreach \x in {1,2,3,6} {
 	  \filldraw[fill=black, draw=black] (3,\x) rectangle (4,\x+1);
 	};
 	\foreach \x in {0,3,5,6} {
 	  \filldraw[fill=black, draw=black] (0,\x) rectangle (1,\x+1);
 	};
 	\foreach \x in {0,2,3,4} {
 	  \filldraw[fill=black, draw=black] (4,\x) rectangle (5,\x+1);
 	};
 	\foreach \x in {0,1,4,6} {
 	  \filldraw[fill=black, draw=black] (1,\x) rectangle (2,\x+1);
 	};
 	\foreach \x in {0,1,2,5} {
 	  \filldraw[fill=black, draw=black] (2,\x) rectangle (3,\x+1);
 	};
 	\draw[white,opacity=0.3] (0,0) grid (7,7);
 	\draw \foreach \y in {1,...,7} {
 	  (-.5,\y -.5) node[] {$u_{\y}$}
 	};
 	\draw \foreach \x in {1,...,7} {
 	  (\x -.5,-.5) node[] {$v_{\x}$}
 	};
 \end{tikzpicture}
 \caption{The graph $K_7 \Box K_7$, where $v_i \in V(K_7)$ and $u_i \in V(K_7)$. The black squares from a $\{2\}$-resolving set of $K_7 \Box K_7$.}\label{fig:rook}
\end{figure}

If $\{v\} \cup (V \setminus N(v))$ is a \emph{proper} subset of $S$ for some $v \in V$, then $S$ is a $\{2\}$-resolving set of $K_m \Box K_n$. The proof is straightforward but quite technical. The set $S$ contains almost all vertices of the graph. When the graph $K_m \Box K_n$ is sufficiently large, the condition 2. of Theorem \ref{thm:rookconds} has potential to produce significantly smaller $\{2\}$-resolving sets. To show that a set satisfying 2. is a $\{2\}$-resolving set of $K_m \Box K_n$, we need the following lemma.

\begin{lem}\label{lem:rookthree}
Let $m\geq n \geq 6$ and $S \subseteq V(K_m \Box K_n)$. If each quadruple contains at least one element of $S$, then there exists at most one row and one column that contain at most two elements of $S$.
\end{lem}
\begin{proof}
Suppose to the contrary that there exist some $r,t \in V(K_n)$, $r \neq t$, such that the rows  $R = \{vr \ | \ v \in V(K_m) \}$ and $T = \{vt \ | \ v \in V(K_m) \}$ both contain at most two elements of $S$. Consider the two rows as partitioned into pairs $\{vr,vt\}$, where $v \in V(K_m)$. 
The rows $R$ and $T$ contain at most four elements of $S$ in total. However, we have $m \geq 6$ pairs, and thus there are at least two pairs, say $\{ar,at\}$ and $\{br,bt\}$, that do not contain an element of $S$. Now the quadruple $\{ar,at,br,bt\}$ does not contain an element of $S$, a contradiction. The claim holds for columns by symmetry.
\end{proof}

\begin{thm}\label{thm:rook}
Let $m \geq n \geq 6$ and $S \subseteq V(K_m \Box K_n)$. If each row and column contains at least two elements of $S$ and each quadruple contains at least one element of $S$, then the set $S$ is a $\{2\}$-resolving set of $K_m \Box K_n$.
\end{thm}
\begin{proof}
To prove that $S$ is a 1-solid-resolving set, it suffices to check that \eqref{eq:superb} holds for any $x \in V \setminus S$. To that end, let $x \in V \setminus S$ and $y \in V$, $y \neq x$. Both the row and column that contain $x$ also contain at least two elements of $S$. The closed neighbourhood of $y$ contains all these four elements if and only if $y=x$. Thus, for any $x \in V \setminus S$ and $y \in V$ there exists $s \in S$ such that $d(s,x) < d(s,y)$. According to Theorem \ref{thm:superb}, the set $S$ is a 1-solid-resolving set of $K_m \Box K_n$.

Let us then consider distinct sets $X,Y \subseteq V(K_m \Box K_n)$ such that $|X| = |Y| = 2$. If for some $x \in X \setminus Y$ and $y \in Y \setminus X$ we have $\{x,y\} \cap S \neq \emptyset$, then clearly $\D_S (X) \neq \D_S (Y)$.

Suppose that for some $x \in X \setminus Y$ and $y \in Y \setminus X$ we have $\{x,y\} \cap S = \emptyset$. According to Lemma \ref{lem:rookthree} at least one of $x$ and $y$ has three elements of $S$ on its row or column. Assume without loss of generality that $x$ is on the row $R$ and $R$ contains at least three elements of $S$. 
If $Y \cap R = \emptyset$, then for at least one $s \in S \cap R$ we have $d(s,X) = 1 < d(s,Y)$. Thus, $\D_S (X) \neq \D_S (Y)$.

Suppose that $Y \cap R \neq \emptyset$, and let $y_1 \in Y \cap R$. Since $x \notin Y$, the vertex $y_1$ cannot be on the same column as $x$. The column $C$ that contains $x$ contains at least two elements of $S$, say $c_1,c_2 \in C \cap S$. We have $d(c_1,x) = d(c_2,x) = 1$ and $d(c_1,y_1) = d(c_2,y_1) = 2$. Let $y_2 \in Y$, $y_2 \neq y_1$. If $y_2 \notin C$, then $d(c_1,y_2) = 2$ or $d(c_2,y_2) = 2$, and thus $\D_S (X) \neq \D_S (Y)$. 

Suppose $y_2 \in C$. Only one of $y_1$ and $y_2$ can be in $S$. Suppose $y_1 \in S$. If $y_1 \notin X$, then $\D_S (X) \neq \D_S (Y)$. Suppose $y_1 \in X$. The row $T$ that contains $y_2$ also contains at least two elements of $S$, say $t_1,t_2 \in T \cap S$. Since $y_2 \notin S$, $t_1 \neq y_2$ and $t_2 \neq y_2$, and thus $t_1,t_2 \notin C$. Now $d(t_1,X) = 2$ or $d(t_2, X) = 2$ since only one of $t_1$ and $t_2$ can be on the same column as $y_1$. Thus, $\D_S (X) \neq \D_S (Y)$.
Similarly, if $y_2 \in S$, we can prove that there is  a vertex $s \in S$ in the same column as $y_1$ such that $d(s,y_1) = 1 < d(s,X)$. 

Suppose $y_1 \notin S$ and $y_2 \notin S$. If the element $x' \in X \setminus \{x\}$ is in the intersection of the column containing $y_1$ and the row containing $y_2$, the elements $x$, $x'$, $y_1$ and $y_2$ form a quadruple. According to our assumption one of these elements is in $S$, and consequently $\D_S (X) \neq \D_S (Y)$. If $x'$ is not on the same column as $y_1$ or on the same row as $y_2$ (both of which contain two elements of $S$), we clearly have $\D_S (X) \neq \D_S (Y)$.
\end{proof}

According to Theorem \ref{thm:rook}, the set illustrated as black squares in Figure \ref{fig:rook} is a $\{2\}$-resolving set of $K_7 \Box K_7$.
The following theorem can be used to obtain a lower bound on the $\{2\}$-metric dimension of $K_m \Box K_n$. Indeed, the left side of Equation~\eqref{EqThmRookLowerBound} decreases as the size of the $\{2\}$-resolving set $S$ increases. Thus, this gives a lower bound on $|S|$.
\begin{thm} \label{ThmRookL2LowerBound}
Let $m \geq n \geq 2$. If $S$ is a $\{2\}$-resolving set of $K_m \Box K_n$, and $q$ and $r$ are integers such that $|S| = qm + r$ with $0 \leq r < m$, then
\begin{equation} \label{EqThmRookLowerBound}
r\binom{n-(q+1)}{2} + (m-r)\binom{n-q}{2} \leq \binom{n}{2} \text{.}
\end{equation}
\end{thm}
\begin{proof}
Assume first that $S$ is an arbitrary subset of $V(K_m \Box K_n)$ and $q$ and $r$ are integers such that $|S| = qm + r$ with $0 \leq r < m$. Denote the columns of $K_m \Box K_n$ by $C_1, \ldots, C_m$. For $i = 1, \ldots, m$, let $x_i$ be the number of elements of $S$ in the column $C_i$, i.e., $x_i = |S \cap C_i|$. Using this notation, each column $C_i$ contains $\binom{n-x_i}{2}$ pairs of vertices not belonging to $S$. Furthermore, the number of such pairs of vertices over all the columns is equal to
\begin{equation} \label{EqRookLowerBound}
\sum_{i=1}^m \binom{n-x_i}{2} \text{.}
\end{equation}

Assume that the set $S'$ gives the minimum value of the sum \eqref{EqRookLowerBound} among the sets with $|S|$ elements. Let us then show that no column of $S'$ contains less than $q$ elements. Suppose to the contrary that there is a column $C_i$ with $|S' \cap C_i| = k_1 < q$. Since $|S'| = qm + r$, there exists a column $C_j$ with $|S' \cap C_j| = k_2 \geq q + 1$. Now we have
\[
\binom{k_1}{2} + \binom{k_2}{2} = \binom{k_1}{2} + \binom{k_2-1}{2} + (k_2-1) > \binom{k_1}{2} + k_1 + \binom{k_2-1}{2} = \binom{k_1+1}{2} + \binom{k_2-1}{2} \text{.}
\]
Hence, the elements of $S'$ in the columns $C_i$ and $C_j$ can be redistributed to obtain a set with the same number of elements as $S'$ and with a smaller sum~\eqref{EqRookLowerBound} (a contradiction). Similarly, it can be shown that no column contains at least $q+2$ elements of $S'$. Indeed, if such a column, say $C_i$ with $|S' \cap C_i| = k_1 \geq q+2$, exists, then there is a column $C_j$ with $|S' \cap C_j| = k_2 \leq q$ and as above we have
\[
\binom{k_1}{2} + \binom{k_2}{2} = \binom{k_1-1}{2} + (k_1-1) + \binom{k_2}{2}  > \binom{k_1-1}{2} + \binom{k_2}{2} + k_2 = \binom{k_1-1}{2} + \binom{k_2+1}{2}
\]
leading to a contradiction. Hence, we may assume that each column contains at least $q$ and at most $q+1$ elements of $S'$. Therefore, as $|S'| = |S| = qm + r$, there exist $r$ columns containing $q+1$ elements and $m-r$ columns containing $q$ elements of $S'$. Thus, we obtain that
\[
\sum_{i=1}^m \binom{n-x_i}{2} \geq r\binom{n-(q+1)}{2} + (m-r)\binom{n-q}{2} \text{.}
\]
Observe that the right side of this inequality decreases as the number of elements of $S$ increases.

Assume then that $S$ is a $\{2\}$-resolving set of $K_m \Box K_n$ (instead of being arbitrary). Now, due to Lemma \ref{lem:quadruple}, no two columns have two same rows without elements of $S$. This implies (by the pigeon hole principle) that
\[
r\binom{n-(q+1)}{2} + (m-r)\binom{n-q}{2} \leq \binom{n}{2} \text{.}
\]
Thus, the claim follows.
\end{proof}

The conditions of Theorem \ref{thm:rook} can also be interpreted as a certain type of design as explained in the following remark. For more on combinatorial designs, see \cite{CombDes} (specifically, parts I and IV).
\begin{remark}
Let $X$ be a set with $n$ elements and $\mathcal{B}$ be a collection of $m$ subsets called \emph{blocks} of $X$ such that (i) any block has at most $n-2$ elements, (ii) each element of $X$ is included in at most $m-2$ blocks and (iii) any pair of elements of $X$ is included in at most one block of $\mathcal{B}$. Each block of $\mathcal{B}$ represents a column of $K_m \Box K_n$; more precisely, the elements of a block correspond to the elements of a column not belonging to $S$. Observe that maximizing the total number of elements in the blocks of $\mathcal{B}$ minimizes the corresponding $\{2\}$-resolving set $S$ of $K_m \Box K_n$. Although the designs satisfying (i), (ii) and (iii) have not earlier been studied, some usual designs work nicely for our purposes:
\begin{itemize}
\item Let $n=m=7$ and $X=\{1, \ldots, 7\}$. A collection $\mathcal{B}_1 = \{\{1,2,4\}, \{1,3,7\}, \{1,5,6\}, \linebreak \{2,3,5\}, \{2,6,7\}, \{3,4,6\}, \{4,5,7\}\}$ is a (balanced incomplete block) design such that each block has $3$ elements, each element is included in $3$ blocks and any pair of elements of $X$ is included in exactly one block of $\mathcal{B}_1$. When we interpret $\mathcal{B}_1$ as explained above, we obtain a $\{2\}$-resolving set of $K_7 \Box K_7$ with 28 elements (see Figure \ref{fig:rook}). Moreover, by Theorem~\ref{ThmRookL2LowerBound}, no smaller $\{2\}$-resolving set exists. Hence, we have $\beta_2(K_7 \Box K_7) = 28$.
\item Let $n = 10$, $m = 12$ and $X = \{1, \ldots, 10\}$. A collection $\mathcal{B}_2 = \{\{1,2,3,4\}, \{1,5,6,7\}, \linebreak \{1,8,9,10\}, \{2,5,8\}, \{2,6,9\}, \{2,7,10\}, \{3,5,10\}, \{3,6,8\}, \{3,7,9\}, \{4,5,9\}, \{4,6,10\}, \linebreak \{4,7,8\}\}$ is a (pairwise balanced) design such that each block has $3$ or $4$ elements, each element is included in $3$ or $4$ blocks and any pair of elements of $X$ is included in exactly one block of $\mathcal{B}_2$. Hence, we obtain a $\{2\}$-resolving set of $K_{10} \Box K_{12}$ with 81 elements. Therefore, by Theorem~\ref{ThmRookL2LowerBound}, we have $\beta_2(K_{10} \Box K_{12}) = 81$.
\end{itemize}

Analogously, any $\{2\}$-resolving set $S$ of $K_m \Box K_n$ can be interpreted as a certain type of design. Indeed, construct a design with $m$ blocks each formed by the elements of a column not belonging to $S$. By Lemma \ref{lem:quadruple}, each such design satisfies the previous condition~(iii) and some other minor constraints depending on whether 1. or 2. of Theorem \ref{thm:rookconds} holds. 
\end{remark}

\section{Flower Snarks}

Flower snarks were first introduced by Isaacs in \cite{IsaacsSnark}. Flower snarks were one the first infinite graph families of 3-regular graphs proven to have no proper 3-edge-coloring. In \cite{ImranFlower}, flower snarks were shown to have a constant $\{1\}$-metric dimension. Let us define flower snarks with the following construction.

\begin{const*}
Let $n = 2k+1$ be an odd integer, $n \geq 5$.
\begin{enumerate}[topsep=0pt,itemsep=0pt]
\item First we draw $n$ copies of the star $K_{1,3}$. We denote by $T_i = \{ a_i, b_i, c_i, d_i \}$ the vertices of the $i$th star, where the leaves of the star are $a_i$, $c_i$ and $d_i$.

\item We connect the vertices $a_i$ by drawing the cycle $a_1 a_2 \ldots a_n a_1$. 

\item We connect the remaining leaves of the stars by drawing the cycle $c_1 c_2 \ldots c_n d_1 d_2 \ldots d_n c_1$. 
\end{enumerate}
The resulting graph is the flower snark $J_n$ with $4n$ vertices.
\end{const*}

Probably the most common way to draw a flower snark is illustrated in Figure \ref{fig:J5} for $J_5$. The graph $J_5$ (and all flower snarks in general) can be drawn as in Figure \ref{fig:J5other}. From this figure it is easy to see that the graph has many automorphisms and that the vertices $c_i$ and $d_i$ do not have any essential differences.

Any shortest path from $v\in T_i$ to $u \in T_j$ can be divided into three parts; the parts inside $T_i$ and $T_j$, and the part from $T_i$ to $T_j$. 
The part from $T_i$ to $T_j$ is usually the obvious, except for $c_1$ and $c_{k+2}$ (and isomorphic cases). For example, one shortest path between $b_1$ and $b_4$ in $J_5$ is $b_1a_1a_5a_4b_4$. However, the unique shortest path between $c_1$ and $c_4$ is $c_1c_2c_3c_4$.

In \cite{ImranFlower}, it was shown that $\beta_1 (J_n) = 3$ when $n \geq 5$. However, the proof for the upper bound $\beta_1 (J_n) \leq 3$ is erroneous. The authors claim that the set $W = \{c_1,d_1,d_k\}$ is a resolving set of $J_n$ since all vertices have unique distance arrays with respect to $W$. However, we have $\D_W (a_1) = (2,2,k+1) = \D_W (b_n)$ and $\D_W (a_k) = (k+1,k+1,2) = \D_W (b_{k+1})$. Thus, the set $W$ is not a resolving set of $J_n$. Despite this, their result holds. We can replace $d_k$ with $d_{k+1}$ in $W$, after which it is straightforward to correct the proof and verify that the new set is indeed a resolving set of $J_n$.

Our goal is to determine the $\ell$-solid- and $\{\ell\}$-metric dimensions of flower snarks. To that end, we first consider the forced vertices of flower snarks. Consider any flower snark $J_n$. Since $n \geq 5$, $J_n$ is a 3-regular graph of girth at least 5. Now, for all $v\in V$ and $U \subseteq V$, $v \notin U$, if $N(v) \subseteq N[U]$, then the set $U$ has at least three elements. Thus, no vertex of $J_n$ is forced for $\{\ell_1\}$-resolving sets or $\ell_2$-solid-resolving sets where $\ell_1 \leq 3$ and $\ell_2 \leq 2$. For all other $\ell$-solid- and $\{\ell\}$-resolving sets all vertices are forced vertices; for all $v\in V$ we can choose $U=N(v)$, and we naturally have $N(v) \subseteq N[U]$. Thus, we have the following theorem.

\begin{thm}
Let $n$ be an odd integer, $n\geq 5$. We have $\beta_\ell (J_n) = 4n$ when $\ell \geq 4$ and $\beta_\ell^s (J_n) = 4n$ when $\ell \geq 3$.
\end{thm}

As for the remaining metric dimensions, we begin by considering $\{3\}$-resolving sets since, quite surprisingly, the difficulty of the proofs increases as the value of $\ell$ decreases.

\begin{figure}
\centering
\begin{subfigure}[b]{0.3\linewidth}
 \centering
 \begin{tikzpicture}[scale=0.85]
  \def\da{.6};
  \def\db{\da+\da};
  \def\dc{\db+\da};
  \def\ofs{12};
  \def\loos{1.5};
  \coordinate (a1) at (90:\da);
  \coordinate (a2) at (18:\da);
  \coordinate (a3) at (306:\da);
  \coordinate (a4) at (234:\da);
  \coordinate (a5) at (162:\da);
  \coordinate (b1) at (90:\db);
  \coordinate (b2) at (18:\db);
  \coordinate (b3) at (306:\db);
  \coordinate (b4) at (234:\db);
  \coordinate (b5) at (162:\db);
  \coordinate (c1) at (90+\ofs:\dc);
  \coordinate (c2) at (18+\ofs:\dc);
  \coordinate (c3) at (306+\ofs:\dc);
  \coordinate (c4) at (234+\ofs:\dc);
  \coordinate (c5) at (162+\ofs:\dc);
  \coordinate (d1) at (90-\ofs:\dc);
  \coordinate (d2) at (18-\ofs:\dc);
  \coordinate (d3) at (306-\ofs:\dc);
  \coordinate (d4) at (234-\ofs:\dc);
  \coordinate (d5) at (162-\ofs:\dc);
  \draw \foreach \x in {1,...,5} \foreach \y in {a,c,d} {
  		(b\x) -- (\y\x) };
  \draw (a1) -- (a2) -- (a3) -- (a4) -- (a5) -- (a1);
  \draw {
  		(c1) to[out=60,in=70,looseness=\loos] (c2) to[out=348,in=358,looseness=\loos] (c3)
  		to[out=276,in=286,looseness=\loos] (c4) to[out=204,in=214,looseness=\loos] (c5)
  		to[out=125,in=120,looseness=\loos] (d1) to[out=35,in=55,looseness=\loos] (d2)
  		to[out=323,in=343,looseness=\loos] (d3) to[out=251,in=271,looseness=\loos] (d4)
  		to[out=179,in=199,looseness=\loos] (d5) -- (c1)
  };
  \draw \foreach \x in {1,...,5} \foreach \y in {a,b,c,d} {
		(\y\x) node[circle, draw, fill=white, inner sep=0pt, minimum width=4pt] {} };
  \draw {
  		(a1) node[right] {$a_1$}
  		(a2) node[below] {$\quad a_2$}
  		(a3) node[below] {$a_3 \ $}
  		(a4) node[left] {$a_4$}
  		(a5) node[above] {$a_5$}
  		(b1) node[right] {$b_1$}
  		(b2) node[below] {$\, b_2$}
  		(b3) node[above] {$\ \, b_3$}
  		(b4) node[right] {$b_4$}
  		(b5) node[above] {$\ b_5$}
  		(c1) node[above] {$c_1 \ $}
  		(c2) node[left] {$c_2$}
  		(c3) node[above] {$c_3$}
  		(c4) node[right] {$c_4$}
  		(c5) node[below] {$\, c_5$}
  		(d1) node[above] {$\ d_1$}
  		(d2) node[below] {$d_2$}
  		(d3) node[left] {$d_3$}
  		(d4) node[above] {$d_4$}
  		(d5) node[above] {$d_5$}
  };
 \end{tikzpicture}
 \caption{The graph $J_5$.}\label{fig:J5}
\end{subfigure}
\hfill
\begin{subfigure}[b]{0.22\linewidth}
 \centering
 \begin{tikzpicture}[scale=0.75]
  \def\da{1};
  \def\db{2};
  \def\dc{\db};
  \def\ofs{24};
  \def\loos{1.5};
  \coordinate (a1) at (90:\da) {};
  \coordinate (a2) at (18:\da) {};
  \coordinate (a3) at (306:\da) {};
  \coordinate (a4) at (234:\da) {};
  \coordinate (a5) at (162:\da) {};
  \coordinate (b1) at (90:\db) {};
  \coordinate (b2) at (18:\db) {};
  \coordinate (b3) at (306:\db) {};
  \coordinate (b4) at (234:\db) {};
  \coordinate (b5) at (162:\db) {};
  \coordinate (c1) at (90+\ofs:\dc) {};
  \coordinate (c2) at (18-\ofs:\dc) {};
  \coordinate (c3) at (306+\ofs:\dc) {};
  \coordinate (c4) at (234-\ofs:\dc) {};
  \coordinate (c5) at (162+\ofs:\dc) {};
  \coordinate (d1) at (90-\ofs:\dc) {};
  \coordinate (d2) at (18+\ofs:\dc) {};
  \coordinate (d3) at (306-\ofs:\dc) {};
  \coordinate (d4) at (234+\ofs:\dc) {};
  \coordinate (d5) at (162-\ofs:\dc) {};
  \draw \foreach \x in {1,...,5} \foreach \y in {a,c,d} {
  		(b\x) -- (\y\x) };
  \draw (a1) -- (a2) -- (a3) -- (a4) -- (a5) -- (a1);
  \draw (c1) -- (c2) -- (c3) -- (c4) -- (c5) -- (d1) -- (d2) -- (d3) -- (d4) -- (d5) -- (c1);
  \draw \foreach \x in {1,...,5} \foreach \y in {a,b,c,d} {
		(\y\x) node[circle, draw, fill=white, inner sep=0pt, minimum width=4pt] {} };
  \draw (0,-3) node[] {};
  \draw{
  		(90:\da-0.35) node[] {$a_1$}
  		(18:\da-0.35) node[] {$a_2$}
  		(306:\da-0.35) node[] {$a_3$}
  		(234:\da-0.35) node[] {$a_4$}
  		(162:\da-0.35) node[] {$a_5$}
  		(b1) node[above] {$b_1$}
  		(b2) node[right] {$b_2$}
  		(b3) node[below] {$b_3$}
  		(b4) node[below] {$b_4$}
  		(b5) node[left] {$b_5$}
  		(c1) node[above] {$c_1$}
  		(c2) node[right] {$c_2$}
  		(c3) node[right] {$c_3$}
  		(c4) node[left] {$c_4$}
  		(c5) node[left] {$c_5$}
  		(d1) node[above] {$d_1$}
  		(d2) node[above] {$d_2$}
  		(d3) node[below] {$d_3$}
  		(d4) node[below] {$d_4$}
  		(d5) node[above] {$d_5$}
  };
 \end{tikzpicture}
 \caption{The graph $J_5$.}\label{fig:J5other}
\end{subfigure}
\hfill
\begin{subfigure}[b]{0.40\linewidth}
 \centering
 \begin{tikzpicture}[scale=0.65]
    \def\xd{1.8};
    \def\yd{1};
    \def\cd{.5}
 	\coordinate (a1) at (-1.8,0);
 	\coordinate (a2) at ($(a1)+(\xd,0)$);
 	\coordinate (a3) at ($(a2)+(\xd,0)$);
 	\coordinate (b1) at ($(a1)+(0,\yd)$);
 	\coordinate (b2) at ($(a2)+(0,\yd)$);
 	\coordinate (b3) at ($(a3)+(0,\yd)$);
 	\coordinate (c1) at ($(b1)+(-\cd,\yd)$);
 	\coordinate (c2) at ($(b2)+(-\cd,\yd)$);
 	\coordinate (c3) at ($(b3)+(-\cd,\yd)$);
 	\coordinate (d1) at ($(b1)+(\cd,\yd)$);
 	\coordinate (d2) at ($(b2)+(\cd,\yd)$);
 	\coordinate (d3) at ($(b3)+(\cd,\yd)$);
    \draw \foreach \x in {1,...,3} \foreach \y in {a,c,d} {
    	(b\x) -- (\y\x) };
    \draw (-1.8-\cd,0) -- (\xd + \cd,0);
    \draw (-1.8-2*\cd,\cd+2*\yd) to[out=330,in=130,looseness=1.5] (c1) 
    	to[out=50,in=130,looseness=1.5] (c2) to[out=50,in=130,looseness=1.5] (c3) 
    	to[out=50,in=140,looseness=1.4] (\xd + 2*\cd,.5*\cd+2*\yd);
    \draw (-1.8-2*\cd,.5*\cd+2*\yd) to[out=40,in=130,looseness=1.4] (d1) 
    	to[out=50,in=130,looseness=1.5] (d2) to[out=50,in=130,looseness=1.5] (d3) 
    	to[out=50,in=210,looseness=1.5] (\xd + 2*\cd,\cd+2*\yd);
    \draw \foreach \x in {1,...,3} \foreach \y in {a,b,c,d} {
		(\y\x) node[circle, draw, fill=white, inner sep=0pt, minimum width=5pt] {} };
	\draw \foreach \x in {(b1),(b3)} {
		\x node[circle, draw, fill=black, inner sep=0pt, minimum width=5pt] {} };
	\draw \foreach \x in {(a2),(c2),(d2)} {
		\x node[circle, draw, fill=gray!60, inner sep=0pt, minimum width=5pt] {} };
    \draw {
    	(0,-.4) node[] {$a_i$}
    	(b1) node[right] {$\, b_{i-1}$}
    	(b2) node[right] {$\, b_{i}$}
    	(b3) node[right] {$\, b_{i+1}$}
    	(c2) node[below] {$c_i \quad $}
    	(d2) node[below] {$\quad d_i$}
    	(-1.8-\cd,0) node[left] {$\cdots$}
    	(\xd + \cd,0) node[right] {$\cdots$}
    	(-1.8-2*\cd,.75*\cd+2*\yd) node[left] {$\cdots$}
    	(\xd + 2*\cd,.75*\cd+2*\yd) node[right] {$\cdots$}
    };
 \end{tikzpicture}
 \caption{A portion of $J_n$.}\label{fig:sets}
\end{subfigure}
\caption{ }
\end{figure}

\subsection{The $\{3\}$-Metric Dimension of $J_n$}

We begin by proving two technical lemmas. In these lemmas, we consider certain sets of vertices with at most three elements. Any $\{3\}$-resolving set should be able to distinguish these sets from each other. However, as we will see, there are very few vertices able to do that. In Figure \ref{fig:sets}, we have illustrated a part of a flower snark, which will help in visualising the sets of vertices discussed in the lemmas. Notice that if $i=1$, then $b_{i-1} = b_n$, and if $i=n$, then $b_{i+1} = b_1$.

\begin{lem}\label{lem:borthree}
Let $i \in \{1,\ldots, n\}$ and 
\begin{align*}
B = \{ b_{i-1} , b_{i+1} \}, && X = B \cup \{ a_i \}, && Y = B \cup \{ c_i \} , && Z = B \cup \{ d_i \}.
\end{align*}
We have
\begin{enumerate}[topsep=3pt,itemsep=1pt]
\item[(i)] $d(s,X) \neq d(s,B)$ if and only if $s \in \{ a_i,b_i \}$,
\item[(ii)] $d(s,Y) \neq d(s,B)$ if and only if $s \in \{ c_i,b_i \}$,
\item[(iii)] $d(s,Z) \neq d(s,B)$ if and only if $s \in \{ d_i,b_i \}$.
\end{enumerate}
\end{lem}
\begin{proof}
Let $v \in V \setminus T_i$ and $u \in T_i$. Any shortest path $v-u$ goes through either $T_{i-1}$ or $T_{i+1}$. Thus, either $d(v,u) \geq d(v,b_{i-1})$ or $d(v,u) \geq d(v,b_{i+1})$, and we have $d(v, X) = d(v, Y) = d(v, Z) = \min \{d(v,b_{i-1}), d(v,b_{i+1}) \} = d(v, B)$.

Consider then the elements of $T_i$. The distances from each element $s \in T_i$ to each of the sets $B$, $X$, $Y$ and $Z$ are presented in the following table:
\begin{align*}
\begin{array}{c||c|c|c|c}
s & d(s,B) & d(s,X) & d(s,Y) & d(s,Z)  \\ \hline \hline
b_i & 3 & 1 & 1 & 1 \\ \hline
a_i & 2 & 0 & 2 & 2 \\ \hline
c_i & 2 & 2 & 0 & 2 \\ \hline
d_i & 2 & 2 & 2 & 0 \\ \hline
\end{array}
\end{align*}
\end{proof}

\begin{lem}\label{lem:3vs3}
Let $i \in \{1,\ldots, n\}$ and 
\begin{align*}
 X = \{ a_i, b_{i-1}, b_{i+1} \}, && Y = \{ c_i, b_{i-1}, b_{i+1} \} , && Z = \{ d_i, b_{i-1}, b_{i+1} \}.
\end{align*}
We have
\begin{enumerate}[topsep=3pt,itemsep=1pt]
\item[(i)] $d(s,X) \neq d(s,Y)$ if and only if $s \in \{ a_i,c_i \}$,
\item[(ii)] $d(s,X) \neq d(s,Z)$ if and only if $s \in \{ a_i,d_i \}$,
\item[(iii)] $d(s,Y) \neq d(s,Z)$ if and only if $s \in \{ c_i,d_i \}$.
\end{enumerate}
\end{lem}
\begin{proof}
Follows from the proof of Lemma \ref{lem:borthree}.
\end{proof}

In the following theorem, the exact values of $\beta_3 (J_n)$ are determined for all $n \geq 5$.

\begin{thm}\label{thm:FS3}
Let $n \geq 5$ be an odd integer. We have $\beta_3 (J_n) = 3n$.
\end{thm}
\begin{proof}
$\beta_3 (J_n) \geq 3n$: Let $S$ be a $\{3\}$-resolving set of $J_n$. Any set of two vertices of $T_i$ contains an element of $S$ by Lemmas \ref{lem:borthree} and \ref{lem:3vs3}. Thus, $|S \cap T_i| \geq 3$ for all $i = 1, \ldots, n$, and the lower bound $\beta_3 (J_n) \geq 3n$ follows.

$\beta_3 (J_n) \leq 3n$: Let $S = V \setminus \{b_i \ | \ i = 1,\ldots, n \}$ (see Figure \ref{fig:FS3}) and let $X \subseteq V$ such that $|X| \leq 3$. We will prove that $S$ is a $\{3\}$-resolving set of $J_n$ by showing how to determine the elements of $X$ when we know the distance array $\D_S (X)$.

If for some $s \in S$ we have $d(s,X) = 0$, then clearly $s \in X$. Thus, if $\D_S (X)$ has three zeros, we have found all elements of $X$ since $|X| \leq 3$.

Assume that $\D_S (X)$ has at most two zeros. We need to determine whether $b_i \in X$ for any $i \in \{1,\ldots,n\}$. Consider any $T_i$. If $d(s,X) \geq 2$ for some $s \in T_i \cap S$, then clearly $b_i \notin X$. If $d(s,X) \leq 1$ for all $s \in T_i \cap S$, then $b_i \in X$. Indeed, assume to the contrary that $b_i \notin X$. There is an element of $X$ in $N(s) \setminus \{b_i\}$ for every $s \in T_i \cap S$ such that $d(s,X) =1$. However, all neighbours of $s$ other than $b_i$ are also in $S$. Since $N(s) \cap N(s') = \{ b_i \}$ for all distinct $s,s' \in T_i \cap S$, there must be at least three zeros in the distance array $\D_S (X)$, a contradiction. Therefore, when $\D_S (X)$ has at most two zeros $b_i \in X$ if and only if $d(s,X) \leq 1$ for all $s \in T_i \cap S$.
\end{proof}

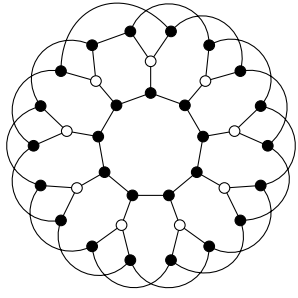
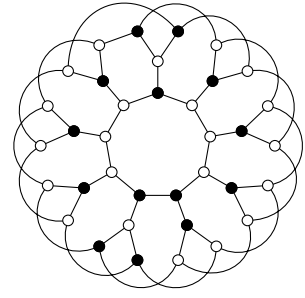
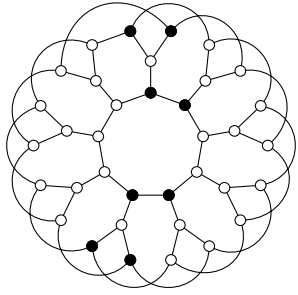
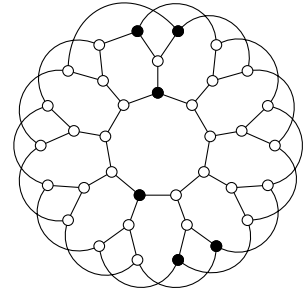
\begin{figure}
\centering
\begin{subfigure}[b]{0.45\linewidth} 
 \centering
 \begin{tikzpicture}[scale=0.7]
  \def\da{1};
  \def\db{\da+.6};
  \def\dc{\db+.6};
  \def\dg{40};
  \def\dgf{90+\dg};
  \def\ofs{10};
  \def\loos{1.5};
  \foreach \x in {1,...,9} \coordinate (a\x) at (\dgf-\x*\dg:\da);
  \foreach \x in {1,...,9} \coordinate (b\x) at (\dgf-\x*\dg:\db);
  \foreach \x in {1,...,9} \coordinate (c\x) at (\dgf-\x*\dg+\ofs:\dc);
  \foreach \x in {1,...,9} \coordinate (d\x) at (\dgf-\x*\dg-\ofs:\dc);
  \draw \foreach \x in {1,...,9} \foreach \y in {a,c,d} {
  		(b\x) -- (\y\x) };
  \draw (a1) -- (a2) -- (a3) -- (a4) -- (a5) -- (a6) -- (a7) -- (a8) -- (a9) -- (a1);
  \draw {
  		(c1) to[out=70,in=90,looseness=\loos] (c2) to[out=30,in=50,looseness=\loos] (c3)
  		to[out=350,in=10,looseness=\loos] (c4) to[out=310,in=330,looseness=\loos] (c5)
  		to[out=270,in=290,looseness=\loos] (c6) to[out=230,in=250,looseness=\loos] (c7)
  		to[out=190,in=210,looseness=\loos] (c8) to[out=150,in=170,looseness=\loos] (c9)
  		to[out=90,in=130,looseness=\loos]
  		(d1) to[out=50,in=75,looseness=\loos] (d2) to[out=10,in=40,looseness=\loos] (d3)
  		to[out=330,in=0,looseness=\loos] (d4) to[out=290,in=320,looseness=\loos] (d5)
  		to[out=250,in=275,looseness=\loos] (d6) to[out=205,in=230,looseness=\loos] (d7)
  		to[out=165,in=195,looseness=\loos] (d8) to[out=125,in=155,looseness=\loos] (d9) -- (c1)
  };
  \draw \foreach \x in {1,...,9} \foreach \y in {b} {
		(\y\x) node[circle, draw, fill=white, inner sep=0pt, minimum width=4pt] {} };
  \draw \foreach \x in {1,...,9} \foreach \y in {a,c,d} {
		(\y\x) node[circle, draw, fill=black, inner sep=0pt, minimum width=4pt] {} };
 \end{tikzpicture}
 \caption{A $\{3\}$-resolving set.}\label{fig:FS3}
\end{subfigure}
\hfill
\begin{subfigure}[b]{0.45\linewidth} 
 \centering
 \begin{tikzpicture}[scale=0.7]
  \def\da{1};
  \def\db{\da+.6};
  \def\dc{\db+.6};
  \def\dg{40};
  \def\dgf{90+\dg};
  \def\ofs{10};
  \def\loos{1.5};
  \foreach \x in {1,...,9} \coordinate (a\x) at (\dgf-\x*\dg:\da);
  \foreach \x in {1,...,9} \coordinate (b\x) at (\dgf-\x*\dg:\db);
  \foreach \x in {1,...,9} \coordinate (c\x) at (\dgf-\x*\dg+\ofs:\dc);
  \foreach \x in {1,...,9} \coordinate (d\x) at (\dgf-\x*\dg-\ofs:\dc);
  \draw \foreach \x in {1,...,9} \foreach \y in {a,c,d} {
  		(b\x) -- (\y\x) };
  \draw (a1) -- (a2) -- (a3) -- (a4) -- (a5) -- (a6) -- (a7) -- (a8) -- (a9) -- (a1);
  \draw {
  		(c1) to[out=70,in=90,looseness=\loos] (c2) to[out=30,in=50,looseness=\loos] (c3)
  		to[out=350,in=10,looseness=\loos] (c4) to[out=310,in=330,looseness=\loos] (c5)
  		to[out=270,in=290,looseness=\loos] (c6) to[out=230,in=250,looseness=\loos] (c7)
  		to[out=190,in=210,looseness=\loos] (c8) to[out=150,in=170,looseness=\loos] (c9)
  		to[out=90,in=130,looseness=\loos]
  		(d1) to[out=50,in=75,looseness=\loos] (d2) to[out=10,in=40,looseness=\loos] (d3)
  		to[out=330,in=0,looseness=\loos] (d4) to[out=290,in=320,looseness=\loos] (d5)
  		to[out=250,in=275,looseness=\loos] (d6) to[out=205,in=230,looseness=\loos] (d7)
  		to[out=165,in=195,looseness=\loos] (d8) to[out=125,in=155,looseness=\loos] (d9) -- (c1)
  };
  \draw \foreach \x in {1,...,9} \foreach \y in {a,b,c,d} {
		(\y\x) node[circle, draw, fill=white, inner sep=0pt, minimum width=4pt] {} };
  \draw \foreach \x in {(a1),(c1),(d1),(b2),(b3),(b4),(b5),(a5),(a6),(c6),(d6),(b7),(b8),(b9)} {
		\x node[circle, draw, fill=black, inner sep=0pt, minimum width=4pt] {} };
 \end{tikzpicture}
 \caption{A 2-solid-resolving set.}\label{fig:FS2s}
\end{subfigure}
\begin{subfigure}[b]{0.45\linewidth} 
 \centering
 \begin{tikzpicture}[scale=0.7]
  \def\da{1};
  \def\db{\da+.6};
  \def\dc{\db+.6};
  \def\dg{40};
  \def\dgf{90+\dg};
  \def\ofs{10};
  \def\loos{1.5};
  \foreach \x in {1,...,9} \coordinate (a\x) at (\dgf-\x*\dg:\da);
  \foreach \x in {1,...,9} \coordinate (b\x) at (\dgf-\x*\dg:\db);
  \foreach \x in {1,...,9} \coordinate (c\x) at (\dgf-\x*\dg+\ofs:\dc);
  \foreach \x in {1,...,9} \coordinate (d\x) at (\dgf-\x*\dg-\ofs:\dc);
  \draw \foreach \x in {1,...,9} \foreach \y in {a,c,d} {
  		(b\x) -- (\y\x) };
  \draw (a1) -- (a2) -- (a3) -- (a4) -- (a5) -- (a6) -- (a7) -- (a8) -- (a9) -- (a1);
  \draw {
  		(c1) to[out=70,in=90,looseness=\loos] (c2) to[out=30,in=50,looseness=\loos] (c3)
  		to[out=350,in=10,looseness=\loos] (c4) to[out=310,in=330,looseness=\loos] (c5)
  		to[out=270,in=290,looseness=\loos] (c6) to[out=230,in=250,looseness=\loos] (c7)
  		to[out=190,in=210,looseness=\loos] (c8) to[out=150,in=170,looseness=\loos] (c9)
  		to[out=90,in=130,looseness=\loos]
  		(d1) to[out=50,in=75,looseness=\loos] (d2) to[out=10,in=40,looseness=\loos] (d3)
  		to[out=330,in=0,looseness=\loos] (d4) to[out=290,in=320,looseness=\loos] (d5)
  		to[out=250,in=275,looseness=\loos] (d6) to[out=205,in=230,looseness=\loos] (d7)
  		to[out=165,in=195,looseness=\loos] (d8) to[out=125,in=155,looseness=\loos] (d9) -- (c1)
  };
  \draw \foreach \x in {1,...,9} \foreach \y in {a,b,c,d} {
		(\y\x) node[circle, draw, fill=white, inner sep=0pt, minimum width=4pt] {} };
  \draw \foreach \x in {a1,a5,a6,a2,c1,c6,d1,d6} {
		(\x) node[circle, draw, fill=black, inner sep=0pt, minimum width=4pt] {} };
 \end{tikzpicture}
 \caption{A $\{2\}$-resolving set.}\label{fig:FS2}
\end{subfigure}
\hfill
\begin{subfigure}[b]{0.45\linewidth} 
 \centering
 \begin{tikzpicture}[scale=0.7]
  \def\da{1};
  \def\db{\da+.6};
  \def\dc{\db+.6};
  \def\dg{40};
  \def\dgf{90+\dg};
  \def\ofs{10};
  \def\loos{1.5};
  \foreach \x in {1,...,9} \coordinate (a\x) at (\dgf-\x*\dg:\da);
  \foreach \x in {1,...,9} \coordinate (b\x) at (\dgf-\x*\dg:\db);
  \foreach \x in {1,...,9} \coordinate (c\x) at (\dgf-\x*\dg+\ofs:\dc);
  \foreach \x in {1,...,9} \coordinate (d\x) at (\dgf-\x*\dg-\ofs:\dc);
  \draw \foreach \x in {1,...,9} \foreach \y in {a,c,d} {
  		(b\x) -- (\y\x) };
  \draw (a1) -- (a2) -- (a3) -- (a4) -- (a5) -- (a6) -- (a7) -- (a8) -- (a9) -- (a1);
  \draw {
  		(c1) to[out=70,in=90,looseness=\loos] (c2) to[out=30,in=50,looseness=\loos] (c3)
  		to[out=350,in=10,looseness=\loos] (c4) to[out=310,in=330,looseness=\loos] (c5)
  		to[out=270,in=290,looseness=\loos] (c6) to[out=230,in=250,looseness=\loos] (c7)
  		to[out=190,in=210,looseness=\loos] (c8) to[out=150,in=170,looseness=\loos] (c9)
  		to[out=90,in=130,looseness=\loos]
  		(d1) to[out=50,in=75,looseness=\loos] (d2) to[out=10,in=40,looseness=\loos] (d3)
  		to[out=330,in=0,looseness=\loos] (d4) to[out=290,in=320,looseness=\loos] (d5)
  		to[out=250,in=275,looseness=\loos] (d6) to[out=205,in=230,looseness=\loos] (d7)
  		to[out=165,in=195,looseness=\loos] (d8) to[out=125,in=155,looseness=\loos] (d9) -- (c1)
  };
  \draw \foreach \x in {1,...,9} \foreach \y in {a,b,c,d} {
		(\y\x) node[circle, draw, fill=white, inner sep=0pt, minimum width=4pt] {} };
  \draw \foreach \x in {(a1),(c1),(d1),(c5),(d5),(a6)} {
		\x node[circle, draw, fill=black, inner sep=0pt, minimum width=4pt] {} };
 \end{tikzpicture}
 \caption{A 1-solid-resolving set.}\label{fig:FS1s}
\end{subfigure}
\caption{Optimal $\{\ell\}$- and $\ell$-solid-resolving sets of $J_9$.}
\end{figure}

\subsection{The 2-Solid-Metric Dimension of $J_n$}

Let $S$ be a 2-solid-resolving set of $J_n$. For any distinct sets $X,Y \subseteq V$ such that $|X| = 2$ and $|Y| \geq 3$, we have $\D_S (X) \neq \D_S (Y)$. In particular, Lemma \ref{lem:borthree} holds for $S$. Thus, either $b_i \in S$ or $\{a_i,c_i,d_i\} \subseteq S$. This observation gives us the obvious lower bound $\beta_2^s (J_n) \geq n$. However, as we will show in Theorem \ref{thm:FS2solid}, the 2-solid-metric dimension of $J_n$ is $n+5$. In order to obtain the lower bound $\beta_2^s (J_n) \geq n+5$, we need the following two lemmas. These lemmas tell us, how many vertices $a_i$, $c_i$ and $d_i$ a 2-solid-resolving set must contain.

Recall that we denote $n=2k+1$, where $k$ is an integer.

\begin{lem}\label{lem:Aatleast3}
Denote $A = \{ a_i \ | \ i = 1, \ldots, n \}$. If a vertex set $S$ is a $2$-solid-resolving set of $J_n$, then there can be at most $k-1$ consecutive elements of $A$ that are not elements of $S$. Consequently, $S$ must contain at least three elements of $A$.
\end{lem}
\begin{proof}
Assume to the contrary that there are $k$ or more consecutive elements of $A$ that are not in $S$. Without loss of generality, we can assume that $\{ a_i \ | \ i = k+2, \ldots, n \} \cap S = \emptyset$.
We will show that the set $S$ is not a 2-solid-resolving set as it does not satisfy \eqref{eq:superb}. To that end, let us consider the vertex $a_n$ and the set $X = \{ c_n,a_1 \}$. For all $b_i$ we have $d(b_i , a_n) = d(b_i, c_n)$, and thus $d(b_i, X) \leq d(b_i, a_n)$. For all $c_i$ we have $d(c_i, c_n) \leq d(a_i,a_n) + 2 = d(c_i,a_n)$. Similarly, we have $d(d_i,c_n) \leq d(d_i,a_n)$ for all $d_i$. Since $d(c_i,X) \leq d(c_i,c_n)$ and $d(d_i,X) \leq d(d_i,d_n)$, we have $d(c_i,X) \leq d(c_i,a_n)$ and $d(d_i,X) \leq d(d_i,a_n)$ for all $i \in \{1, \ldots, n\}$. Let $a_j \in A \cap S$. Since $1 \leq j \leq k+1$, we have $d(a_j,a_1) \leq d(a_j,a_n)$, and thus $d(a_j,X) \leq d(a_j,a_n)$. Consequently, the set $S$ does not satisfy \eqref{eq:superb} for $a_n$ and $X$, and is not a 2-solid-resolving set of $J_n$ according to Theorem \ref{thm:superb}. 

Consequently, $A$ contains at least three elements of $S$.
\end{proof}

We denote the cycle $c_1 c_2 \ldots c_n d_1 d_2 \ldots d_n c_1$ by $C$.

\begin{lem}\label{lem:Catleast4}
If a vertex set $S$ is a $2$-solid-resolving set of $J_n$, then there can be at most $k$ consecutive vertices of $C$ that are not in $S$. Consequently, $S$ must contain at least four elements of $C$.
\end{lem}
\begin{proof}
Assume to the contrary that there are at least $k+1$ consecutive vertices of $C$ that are not in $S$. Without loss of generality, assume that $\{c_i \ | \ i = 1, \ldots, k+1\} \cap S = \emptyset$. 

Consider the vertex $c_1$ and the set $X = \{ d_1, d_n \}$. We will show that $d(s,c_1) \geq d(s,X)$ for all $s \in S$. For all $a_i$, we have $d(a_i,c_1) = d(a_i, a_1) + 2 = d(a_i,d_1)$. Consequently, $d(a_i,X) \leq d(a_i,c_1)$ for all $i \in \{1,\ldots,n\}$. Similarly, we have $d(b_i,c_1) = d(a_i,a_1) + 1 = d(b_i,d_1)$ and $d(b_i,X) \leq d(b_i,c_1)$ for all $i \in \{ 1, \ldots, n\}$. Consider then a vertex $d_j$. If $1 \leq j \leq k$, then $d(d_j,c_1) = d(d_j,d_1) + 2$. If $k+1 \leq j \leq n$, then $d(d_j,c_1) = d(d_j,d_n) + 1$. Thus, $d(d_j,X) < d(d_j,c_1)$ for all $j \in \{ 1, \ldots, n\}$. Similarly, for all $c_j$ where $j \in \{ k+3, \ldots, n \}$ we have $d(c_j,c_1) = d(c_j,d_1) + 2$. Finally, since $d(c_{k+2}, c_1) = k+1 = d(c_{k+2},d_n)$, we have $d(c_j,X) \leq d(c_j,c_1)$ for all $j \in \{k+2,\ldots,n\}$. Now the set $S$ is not a 2-solid-resolving set of $J_n$ since it does not satisfy \eqref{eq:superb} for $c_1$ and $X$.

Consequently, $C$ contains at least four elements of $S$.
\end{proof}

\begin{thm}\label{thm:FS2solid}
Let $n\geq 5$ be an odd integer. We have $\beta_2^s (J_n) = n+5$.
\end{thm}
\begin{proof}
$\beta_2^s (J_n) \geq n+5$: Assume that $S$ is a 2-solid-resolving set of $J_n$ with at most $n+4$ elements. Recall that according to Lemma \ref{lem:borthree} we have either $b_i \in S$ or $\{a_i,c_i,d_i\} \subseteq S$ for all $i \in \{1,\ldots,n\}$. 
According to Lemma \ref{lem:Catleast4} the set $S$ contains at least four elements of $C$. Since $|S| \leq n+4$, the set $S$ has exactly four elements of $C$ due to Lemma \ref{lem:borthree}(i). Now, if $c_i \notin S$ or $d_i \notin S$, then $b_i \in S$ and $a_i \notin S$ since otherwise $S$ would have more than $n+4$ elements. If $c_i$ and $d_i$ are both in $S$, we have either $b_i \in S$ or $a_i \in S$. Since $S$ contains four elements of $C$, there can be at most two elements $a_i$ in $S$. Now, according to Lemma \ref{lem:Aatleast3} the set $S$ is not a 2-solid-resolving set of $J_n$.

$\beta_2^s (J_n) \leq n+5$: Let $$S = \{ a_1,c_1,d_1,a_{k+1},a_{k+2},c_{k+2},d_{k+2} \} \cup \{ b_i \ | \ i \in \{1,\ldots,n\}, i \neq 1,k+2 \}.$$ See Figure \ref{fig:FS2s} for an example of this set.
We have $|S| = 7 + n-2 = n+5$. We will show that $S$ satisfies \eqref{eq:superb} for $\ell = 2$, and is thus a 2-solid-resolving set of $J_n$. Clearly, for all $s \in S$ and $X \subseteq V$ such that $s \notin X$ we have $d(s,s) < d(s,X)$. Consider then the vertices that are not in $S$. We divide the study by the types of the vertices in $J_n$.
\begin{itemize}
 \item[$a_i:$] Assume that $2 \leq i \leq k$, the other case where $k+3 \leq i \leq n$ goes similarly. Since $a_i \notin S$, we have $b_i \in S$. Let $X \subseteq V$, $|X| \leq 2$ and $a_i \notin X$. If $X \cap T_i = \emptyset$, then $d(b_i,a_i) < d(b_i,X)$. Assume then that $X \cap T_i \neq \emptyset$. Observe that $d(a_1,a_i) < d(a_1, X \cap T_i)$ and $d(a_{k+1},a_i) < d(a_{k+1}, X \cap T_i)$. If $X \subseteq T_i$, then $d(a_1,a_i) < d(a_1,X)$ and $d(a_{k+1},a_i) < d(a_{k+1},X)$. Suppose then that $|X \cap T_i| = 1$ and $x \in X \setminus T_i$. If $d(a_1,x) \leq d(a_1,a_i)$ and $d(a_{k+1},x) \leq d(a_{k+1},a_i)$, then $d(a_1,a_{k+1}) \leq d(a_1,x) + d(x,a_{k+1}) \leq d(a_1,a_i) + d(a_i,a_{k+1})$. Since the path $a_1a_2 \ldots a_{k+1}$ is the unique shortest path between $a_1$ and $a_{k+1}$, we have $x=a_j$ for some $j \in \{1, \ldots, k+1\}$, $ j\neq i$. Consequently, either $d(a_1,a_i) < d(a_1,x)$ or $d(a_{k+1},a_i) < d(a_{k+1},x)$. Thus, either $d(a_1,a_i) < d(a_1,X)$ or $d(a_{k+1},a_i) < d(a_{k+1},X)$.
 
 \item[$b_i:$] Since $b_i \notin S$, either $i=1$ or $i=k+2$. Consider the case where $i = 1$ (the case where $i=k+2$ goes similarly). Let $X \subseteq V$, $|X| \leq 2$ and $b_1 \notin X$. If $S$ does not satisfy \eqref{eq:superb}, then $d(a_1,X)$, $d(c_1,X)$ and $d(d_1,X)$ are all at most 1. However, now each of the sets $\{a_1,a_2,a_n\}$, $\{c_1,c_2,d_n\}$ and $\{d_1,d_2,c_n\}$ must contain at least one element of $X$. Since these sets do not intersect, the set $X$ has at least three elements, a contradiction.
 
 \item[$c_i,d_i:$] Consider the vertex $c_i$ where $2 \leq i \leq k+1$ (the other cases go similarly). Let $X \subseteq V$, $|X| \leq 2$ and $c_i \notin X$. Assume that $d(s,X) \leq d(s,c_i)$ for all $s \in S$. Since $d(b_i,X) \leq d(b_i,c_i)$, we have $X \cap \{b_i,d_i,a_i\} \neq \emptyset$. However, for all $v \in \{ b_i, d_i,a_i\}$ we have $d(c_1,c_i) < d(c_1,v)$ and $d(c_{k+2},c_i) < d(c_{k+2},v)$. Thus, $X$ must have an element $x$ such that $d(c_1,x) \leq d(c_1,c_i)$ and $d(c_{k+2},x) \leq d(c_{k+2},c_i)$.
 
 The path $c_1c_2 \ldots c_{k+1}c_{k+2}$ is the unique shortest path between $c_1$ and $c_{k+2}$. Naturally, for all $c_j$, where $j\neq i$, we have either $d(c_1,c_i) < d(c_1,c_j)$ or $d(c_{k+2},c_i) < d(c_{k+2},c_j)$. For all other vertices $v \notin \{ c_1, \ldots , c_{k+2} \}$, we have $d(c_1,v) + d(v,c_{k+2}) > d(c_1, c_{k+2}) = d(c_1,c_i) + d(c_i,c_{k+2})$, and thus $d(c_1,c_i) < d(c_1,v)$ or $d(c_{k+2},c_i) < d(c_{k+2},v)$. Thus, there is no such vertex $x$ that $d(c_1,x) \leq d(c_1,c_i)$ and $d(c_{k+2},x) \leq d(c_{k+2},c_i)$.
\end{itemize}
\end{proof}

\subsection{The $\{2\}$-Metric Dimension of $J_n$}

As we have seen in the two previous sections, the $\{3\}$- and 2-solid-metric dimensions of $J_n$ are dependent on $n$. However, we will see in Theorem \ref{thm:FS2rsub} that the $\{2\}$-metric dimension is at most eight for any $J_n$.

Our computer calculations have shown that $\beta_2 (J_5) = 7$, and $S = \{ a_1,a_3,b_2,b_4,c_1,c_3,d_1 \}$, for example, is a $\{2\}$-metric basis of $J_5$. Our calculations have also shown that $\beta_2 (J_n) = 8$ when $7 \leq n \leq 19$. We will prove the upper bound $\beta_2 (J_n) \leq 8$ in the following theorem, and we conjecture that the lower bound $\beta_2 (J_n) \geq 8$ holds for all $n \geq 7$.

The proof of the following theorem is surprisingly difficult with traditional methods of comparing distance arrays. To show the upper bound $\beta_2 (J_n) \leq 8$ we will construct a $\{2\}$-resolving set of $J_n$ with eight elements. We have verified with a computer that the set we provide is indeed a $\{2\}$-resolving set of $J_n$ when $7 \leq n \leq 19$. To show the claim for $n \geq 21$ we use a reduction-like approach. We will show that if the set was not a $\{2\}$-resolving set of $J_n$ then it would not be a $\{2\}$-resolving set of $J_{n-2}$. The idea behind the proof is that if we carefully remove two stars $T_i$ from $J_n$ and add necessary edges (for example, in Figure \ref{fig:sets}, we can remove the star $T_i$ and connect the stars $T_{i-1}$ and $T_{i+1}$), we obtain $J_{n-2}$, and the distances in $J_n$ and $J_{n-2}$ are highly dependent on each other.

\begin{thm}\label{thm:FS2rsub}
Let $n = 2k+1 \geq 7$. We have $\beta_2 (J_n) \leq 8$.
\end{thm}
\begin{proof}
Denote
\begin{align*}
& I = T_n \cup T_1 \cup T_2 \cup T_3, &&
J = T_k \cup T_{k+1} \cup T_{k+2} \cup T_{k+3}, \\
& I' = I \cup T_{n-1} \cup T_4, &&
J' = J \cup T_{k-1} \cup T_{k+4}, \\
& S_I = \{ a_1,c_1,d_1,a_2 \}, &&
S_J = \{ a_{k+1},a_{k+2},c_{k+2},d_{k+2} \}.
\end{align*}
Let $S = S_I \cup S_J$ (see Figure \ref{fig:FS2}). We will show that the set $S$ is a $\{2\}$-resolving set of $J_n$. It is easy to check with a computer that the set $S$ is a $\{2\}$-resolving set when $7\leq n \leq 19$. 

Assume to the contrary that the set $S$ is not a $\{2\}$-resolving set of $J_n$, where $n\geq 21$, and that the set $S$ is a $\{2\}$-resolving set of $J_{n-2}$. We denote the distance arrays in $J_n$ by $\D_S^n$ and the distance arrays in $J_{n-2}$ by $\D_S^{n-2}$. Consider nonempty sets $X,Y \subseteq V(J_n)$ such that $|X| \leq 2$, $|Y| \leq 2$, $X \neq Y$ and $\D_S^n(X) = \D_S^n (Y)$. It is easy to see that if $\D_{S_I}^n (X)$ contains at least one distance that is at most 2, then we have $X \cap I' \neq \emptyset$. Furthermore, if all distances in $\D_{S_I}^n (X)$ are at least 3, then we have $X \cap I = \emptyset$. The same holds for $S_J$, $J$ and $J'$ by symmetry. 

If both $\D_{S_I}^n (X)$ and $\D_{S_J}^n (X)$ contain at least one distance that is at most 2, we have $X \cap I' \neq \emptyset$ and $X \cap J' \neq \emptyset$. Since $\D_S^n (X) = \D_S^n (Y)$, we have $Y \cap I' \neq \emptyset$ and $Y \cap J' \neq \emptyset$. We may think of $J_{n-2}$ as being obtained from $J_n$ by removing two stars from opposite sides of $J_n$ such that they are halfway between $I$ and $J$. Let $X',Y' \subseteq V(J_{n-2})$ consist of vertices that are in exactly the same positions as the elements of $X$ and $Y$ with respect to $S_I$ and $S_J$. Since $n-2 \geq 19$, we have $d_{J_{n-2}}(s,v) \leq 5 \leq k-4 \leq d_{J_{n-2}}(s,u)$ for all $s \in S_I$, $v \in I'$ and $u \in J'$ (sim. for $s \in S_J$, $v \in J'$ and $u \in I'$). Thus, $\D_{S_I}^{n-2} (X' \cap I') = \D_{S_I}^{n-2} (X')$ and $\D_{S_J}^{n-2} (X' \cap J') = \D_{S_J}^{n-2} (X')$, and the same also holds for $Y'$. Now we have $\D_{S}^{n-2} (X') = \D_S^n(X)= \D_S^n (Y) = \D_S^{n-2} (Y')$. However, $X\neq Y$ implies that $X' \neq Y'$, and since $S$ is a $\{2\}$-resolving set of $J_{n-2}$, we must have $\D_S^{n-2} (X') \neq \D_S^{n-2} (Y')$, a contradiction.

Assume then that all distances in $\D_{S_I}^n (X)$ are at least 3 (the case where this holds for $\D_{S_J}^n(X)$ goes similarly). Now, we have $X \cap I = \emptyset$ and $Y \cap I = \emptyset$. We may think of $J_{n-2}$ as being obtained from $J_n$ by removing the stars $T_3$ and $T_n$. Let $X',Y' \subseteq V(J_{n-2})$ consist of vertices that are in exactly the same positions as the elements of $X$ and $Y$ with respect to $S_J$. Now, we have $(X' \cup Y') \cap (T_1 \cup T_2) = \emptyset$, and thus
\begin{align*}
& \D_{S_I}^{n-2} (X') = \D_{S_I}^n (X) - (1,1,1,1), &&
\D_{S_J}^{n-2} (X') = \D_{S_J}^n (X), \\
& \D_{S_I}^{n-2} (Y') = \D_{S_I}^n (Y) - (1,1,1,1), &&
\D_{S_J}^{n-2} (Y') = \D_{S_J}^n (Y).
\end{align*}
Consequently, $\D_S^{n-2} (X') = \D_S^{n-2} (Y')$ if and only if $\D_S^{n} (X)=\D_S^{n} (Y)$. Since $S$ is a $\{2\}$-resolving set of $J_{n-2}$ and $X' \neq Y'$, we have $\D_S^{n-2} (X') \neq \D_S^{n-2} (Y')$, a contradiction.
\end{proof}

\subsection{The 1-Solid-Metric Dimension of $J_n$}

We begin the section by giving an upper bound on $\beta_1^s (J_n)$ for all $n\geq 5$.

\begin{thm}\label{thm:FS1sub}
 Let $n=2k+1 \geq 5$. We have $\beta_1^s (J_n) \leq 6$.
\end{thm}
\begin{proof}
Let $S = \{ a_1,a_{k+2},c_1,d_1,c_{k+1},d_{k+1} \}$ (see Figure \ref{fig:FS1s}). We will show that the set $S$ is a 1-solid-resolving set of $J_n$ by proving that $S$ satisfies \eqref{eq:superb}. We divide the proof by the types of the vertices of $J_n$.
\begin{itemize}
 \item[$a_i$:] 
 Assume that $i \in \{2,\ldots,k+1\}$. The vertex $a_i$ is along some shortest path from $a_1$ to $c_{k+1}$. If there exists a vertex $v \in V\setminus \{a_i\}$ such that $d(a_1,v)\leq d(a_1,a_i)$ and $d(c_{k+1},v) \leq d(c_{k+1},a_i)$, then $v$ is also along a shortest path from $a_1$ to $c_{k+1}$. Moreover, we have $d(a_1,v) = d(a_1,a_i)$ and $d(c_{k+1},v) = d(c_{k+1},a_i)$, and thus $v \in \{b_{i-1},c_{i-2}\}$. Similarly, if $d(d_{k+1},v) \leq d(d_{k+1},a_i)$, then $v \in \{b_{i-1},d_{i-2}\}$. Thus, we have $v = b_{i-1}$. However, we clearly have $d(a_{k+2},a_i) < d(a_{k+2},b_{i-1})$. Thus, \eqref{eq:superb} is satisfied for all $a_i$ where $i \in \{2,\ldots , k+1\}$. The case where $i \in \{k+3, \ldots, n\}$ goes similarly (look at the shortest paths from $a_{k+2}$ to $c_1$ and $d_1$).

 \item[$b_i$:] 
 Assume that $i \in \{1,\ldots , k+1\}$. By the argument above, the only vertex $v \in V \setminus \{b_i\}$ that is at the same distance from $a_1$, $c_{k+1}$ and $d_{k+1}$ as $b_i$ is $a_{i+1}$. However, since $d(c_1,b_i) < d(c_1,a_{i+1})$ for all $i \in \{1, \ldots , k+1\}$, the condition \eqref{eq:superb} is satisfied for all $b_i$ where $i \in \{1, \ldots , k+1\}$. Similarly, we can prove that \eqref{eq:superb} holds for all $b_i$ where $i \in \{k+2, \ldots , n\}$ by looking at the shortest paths from $a_{k+2}$ to $c_1$ and $d_1$.

 \item[$c_i$, $d_i$:] Each $c_i$ and $d_i$ is along one of the four unique shortest paths: $c_1-c_{k+1}$, $c_{k+1}-d_1$, $d_1-d_{k+1}$ and $d_{k+1}-c_1$. Thus, \eqref{eq:superb} is satisfied for all $c_i$ and $d_i$.
\end{itemize}
\end{proof}

Let $P$ be a shortest path between $u$ and $v$ in $J_n$. We denote $\rho_n (u,v) = t -1$, where $t$ is the number of stars that intersect with $P$. Thus, $\rho_n (u,v)$ is the distance $P$ traverses in order to get from the star that contains $u$ to the star that contains $v$. The distance $d(u,v)$ could now be written as $d(u,v) = \rho_n (u,v) + r$, where $r$ is the distance that $P$ traverses inside the stars that contain $u$ and $v$.

To determine the exact 1-solid-metric dimension of $J_n$ we still need to prove the lower bound $\beta_1^s (J_n) \geq 6$. Computer calculations have shown this lower bound to hold for $5 \leq n \leq 39$. 
The idea behind the proof of the following theorem is to prove that if for some $J_n$ we have $\beta_1^s (J_n) \leq 5$, then we also have $\beta_1^s (J_{n-2}) \leq 5$. 
To that end, we assume that the set $S$, $|S| = 5$, is a 1-solid-resolving set of $J_n$. We then construct $J_{n-2}$ from $J_n$ by removing the stars $T_1$ and $T_{k+1}$ and adding necessary edges (see Figure \ref{fig:FSreduction}). As long as the stars close to the stars that were removed did not contain any elements of $S$ the distances from the elements of $S$ to other vertices behave well and predictably after the removal of the two stars. Then we can construct a 1-solid-resolving set of $J_{n-2}$ from $S$, and we reach a contradiction to the lower bound shown with a computer.

\begin{thm}\label{thm:FS1solid}
Let $n = 2k+1 \geq 5$. We have $\beta_1^s (J_n) = 6$.
\end{thm}
\begin{proof}
Due to Theorem \ref{thm:FS1sub}, it suffices to show the lower bound $\beta_1^s (J_n) \geq 6$. We showed this lower bound for $n\leq 39$ by an exhaustive search with a computer. To prove the claim for all $n\geq 41$ we will show that if for some $n\geq 41$ we have $\beta_1^s (J_n) \leq 5$, then we also have $\beta_1^s (J_{n-2})\leq 5$.

Let $n = 2k+1 \geq 41$ and let $S$ be a 1-solid-resolving set of $J_n$ such that $|S| = 5$. Throughout the proof, we will refer to Figure \ref{fig:FSreduction}, where the flower snarks are smaller than what the proof requires for technical reasons. Consider the set $\{T_i \ | \ i \in \{1,2,k-1,k,k+1,k+2,n-1,n\}\}$ (illustrated with a gray background in Figure \ref{fig:FSreduction} for $J_{21}$) and its isomorphic images. There are $n$ such sets and each $s \in S$ is in eight of these sets. Since $|S| = 5$, at least one of these sets does not contain any elements of $S$ if $n > 8 \cdot 5 = 40$. Since $n\geq 41$, we can assume that the stars $T_i$ where $i \in \{ 1,2,k-1,k,k+1,k+2,n-1,n \}$ do not contain elements of $S$. Let $m = n-2 = 2l+1$. We denote by $T_i^n$ a star in $J_n$ and by $T_i^m$ a star in $J_m$.

Let $\alpha : V(J_n) \rightarrow V(J_m)$ be a surjection such that 
\begin{align*}
\alpha (x_i) = \left\lbrace \begin{array}{ll}
x_1, & \text{if } i = n, \\
x_i, & \text{if } i \in \{ 1, \ldots, k \}, \\
x_{i-1}, & \text{otherwise,}
\end{array} \right.
\end{align*}
where $x_i \in \{ a_i,b_i,c_i,d_i \}$. The image of $x_i$ is the same type as $x_i$, that is, if $x_i = a_i$, then $\alpha (x_i) = a_j$ for some $j \in \{ 1, \ldots, m \}$ and similarly for $x_i = b_i,c_i,d_i$. The preimages of $x_1$ and $x_{l+1}$ are $\alpha^{-1} (x_1) = \{ x_1,x_n \}$ and $\alpha^{-1} (x_{l+1}) = \{ x_k,x_{k+1} \}$ (illustrated as black vertices in Figure \ref{fig:FSreduction}). For all other $x_i \in V(J_m)$ the preimages are unique.

\begin{figure}
\centering
 \begin{tikzpicture}[scale=1]
  \def\nn{21}; 
  \def\da{1.85}; 
  \def\dg{17.14}; 
  \def\dgf{90+\dg}; 
  \def\db{\da+.4}; 
  \def\dc{\db+.4}; 
  \def\ofs{4.5}; 
  \def\loos{1.3};
  \foreach \x in {1,...,\nn} \coordinate (a\x) at (\dgf-\x*\dg:\da);
  \foreach \x in {1,...,\nn} \coordinate (b\x) at (\dgf-\x*\dg:\db);
  \foreach \x in {1,...,\nn} \coordinate (c\x) at (\dgf-\x*\dg+\ofs:\dc);
  \foreach \x in {1,...,\nn} \coordinate (d\x) at (\dgf-\x*\dg-\ofs:\dc);
  \fill[fill=gray!30] (132:\dc+0.45) to[out=38,in=163,looseness=1] (65:\dc+0.45) 
  					-- (65:\da-0.2) to[out=163,in=38,looseness=1] (132:\da-0.2);
  \fill[fill=gray!30] (254:\dc+0.45) to[out=345,in=230,looseness=1] (322:\dc+0.45) 
  					-- (322:\da-0.2) to[out=230,in=345,looseness=1] (254:\da-0.2);
  \draw \foreach \x in {1,...,\nn} \foreach \y in {a,c,d} {
  		(b\x) -- (\y\x) };
  \draw (a1) -- (a2) -- (a3) -- (a4) -- (a5) -- (a6) -- (a7) -- (a8) -- (a9) -- (a10) -- (a11)
  			 -- (a12) -- (a13) -- (a14) -- (a15)  -- (a16) -- (a17) -- (a18) -- (a19) -- (a20) -- (a21) -- (a1);
  \draw {
  		(c1) to[out=70,in=90,looseness=\loos] (c2) to[out=73,in=73,looseness=\loos] (c3)
  		to[out=56,in=56,looseness=\loos] (c4) to[out=39,in=39,looseness=\loos] (c5)
  		to[out=21,in=21,looseness=\loos] (c6) to[out=4,in=4,looseness=\loos] (c7)
  		to[out=347,in=347,looseness=\loos] (c8) to[out=330,in=330,looseness=\loos] (c9)
  		to[out=313,in=313,looseness=\loos] (c10) to[out=295,in=295,looseness=\loos] (c11)
  		to[out=278,in=278,looseness=\loos] (c12) to[out=261,in=261,looseness=\loos] (c13)
  		to[out=244,in=244,looseness=\loos] (c14) to[out=227,in=227,looseness=\loos] (c15)
  		to[out=210,in=210,looseness=\loos] (c16) to[out=192,in=192,looseness=\loos] (c17)
  		to[out=175,in=175,looseness=\loos] (c18) to[out=158,in=158,looseness=\loos] (c19)
  		to[out=141,in=141,looseness=\loos] (c20) to[out=124,in=124,looseness=\loos] (c21)
  		to[out=70,in=130,looseness=\loos]
  		(d1) to[out=70,in=90,looseness=\loos] (d2) to[out=73,in=73,looseness=\loos] (d3)
  		to[out=56,in=56,looseness=\loos] (d4) to[out=39,in=39,looseness=\loos] (d5)
  		to[out=21,in=21,looseness=\loos] (d6) to[out=4,in=4,looseness=\loos] (d7)
  		to[out=347,in=347,looseness=\loos] (d8) to[out=330,in=330,looseness=\loos] (d9)
  		to[out=313,in=313,looseness=\loos] (d10) to[out=295,in=295,looseness=\loos] (d11)
  		to[out=278,in=278,looseness=\loos] (d12) to[out=261,in=261,looseness=\loos] (d13)
  		to[out=244,in=244,looseness=\loos] (d14) to[out=227,in=227,looseness=\loos] (d15)
  		to[out=210,in=210,looseness=\loos] (d16) to[out=192,in=192,looseness=\loos] (d17)
  		to[out=175,in=175,looseness=\loos] (d18) to[out=158,in=158,looseness=\loos] (d19)
  		to[out=141,in=141,looseness=\loos] (d20) to[out=124,in=124,looseness=\loos] (d21)
  		-- (c1)
  };
  \draw \foreach \x in {1,...,\nn} \foreach \y in {a,b,c,d} {
		(\y\x) node[circle, draw, fill=white, inner sep=0pt, minimum width=4pt] {} };
  \draw \foreach \x in {1,21,10,11} \foreach \y in {a,b,c,d} {
		(\y\x) node[circle, draw, fill=black, inner sep=0pt, minimum width=4pt] {} };
  \draw \foreach \x in {(c3),(a6),(d8),(d18),(d19)}{
		\x node[circle, draw, fill=gray!80, inner sep=0pt, minimum width=4pt] {} };
  \draw {
  		(60:\dc+0.5) node[] {$s$}
  		(7:\da-0.3) node[] {$x'$}
  		(328:\dc+0.5) node[] {$t$}
  		(137:\dc+0.5) node[] {$z$}
  		(154:\dc+0.5) node[] {$y'$}
  };
 \draw (140:\dc+1.5) node[] {$J_{21}:$};
 \draw[->] (0:\dc+1) -- (0:\dc+1.8);
 \draw (3:\dc+1.4) node[] {$\alpha$};
 \draw {
 		(90:\dc+0.6) node[] {$T_1$}
 		(107:\dc+0.65) node[] {$T_n$}
 		(282:\dc+0.7) node[] {$T_{k+1}$}
 		(297:\dc+0.7) node[] {$T_k$}
 };
\begin{scope}[shift={(8,0)}]
  \def\nn{19}; 
  \def\da{1.8}; 
  \def\dg{18.95}; 
  \def\dgf{90+\dg}; 
  \def\db{\da+.4}; 
  \def\dc{\db+.4}; 
  \def\ofs{5}; 
  \def\loos{1.3};
  \foreach \x in {1,...,\nn} \coordinate (a\x) at (\dgf-\x*\dg:\da);
  \foreach \x in {1,...,\nn} \coordinate (b\x) at (\dgf-\x*\dg:\db);
  \foreach \x in {1,...,\nn} \coordinate (c\x) at (\dgf-\x*\dg+\ofs:\dc);
  \foreach \x in {1,...,\nn} \coordinate (d\x) at (\dgf-\x*\dg-\ofs:\dc);
  \fill[fill=gray!30] (117:\dc+0.45) to[out=26,in=158,looseness=1] (63:\dc+0.45) 
  					-- (63:\da-0.2) to[out=158,in=26,looseness=1] (117:\da-0.2);
  \fill[fill=gray!30] (250:\dc+0.45) to[out=345,in=220,looseness=1] (308:\dc+0.45) 
  					-- (308:\da-0.2) to[out=220,in=345,looseness=1] (250:\da-0.2);
  \draw[dashed] (80.5:\dc+0.8) -- (80.5:\da-0.4) -- (289:\da-0.4) -- (289:\dc+0.8); 
  \draw[dashed] (99.5:\dc+0.8) -- (99.5:\da-0.4) -- (270:\da-0.4) -- (270:\dc+0.8); 
  \draw \foreach \x in {1,...,\nn} \foreach \y in {a,c,d} {
  		(b\x) -- (\y\x) };
  \draw (a1) -- (a2) -- (a3) -- (a4) -- (a5) -- (a6) -- (a7) -- (a8) -- (a9) -- (a10) -- (a11)
  			 -- (a12) -- (a13) -- (a14) -- (a15)  -- (a16) -- (a17) -- (a18) -- (a19) -- (a1);
  \draw {
  		(c1) to[out=70,in=90,looseness=\loos] (c2) to[out=71,in=71,looseness=\loos] (c3)
  		to[out=52,in=52,looseness=\loos] (c4) to[out=33,in=33,looseness=\loos] (c5)
  		to[out=14,in=14,looseness=\loos] (c6) to[out=355,in=355,looseness=\loos] (c7)
  		to[out=336,in=336,looseness=\loos] (c8) to[out=317,in=317,looseness=\loos] (c9)
  		to[out=298,in=298,looseness=\loos] (c10) to[out=279,in=279,looseness=\loos] (c11)
  		to[out=260,in=260,looseness=\loos] (c12) to[out=241,in=241,looseness=\loos] (c13)
  		to[out=222,in=222,looseness=\loos] (c14) to[out=203,in=203,looseness=\loos] (c15)
  		to[out=184,in=184,looseness=\loos] (c16) to[out=165,in=165,looseness=\loos] (c17)
  		to[out=146,in=146,looseness=\loos] (c18) to[out=127,in=127,looseness=\loos] (c19)
  		to[out=65,in=135,looseness=\loos]
  		(d1) to[out=70,in=90,looseness=\loos] (d2) to[out=71,in=71,looseness=\loos] (d3)
  		to[out=52,in=52,looseness=\loos] (d4) to[out=33,in=33,looseness=\loos] (d5)
  		to[out=14,in=14,looseness=\loos] (d6) to[out=355,in=355,looseness=\loos] (d7)
  		to[out=336,in=336,looseness=\loos] (d8) to[out=317,in=317,looseness=\loos] (d9)
  		to[out=298,in=298,looseness=\loos] (d10) to[out=279,in=279,looseness=\loos] (d11)
  		to[out=260,in=260,looseness=\loos] (d12) to[out=241,in=241,looseness=\loos] (d13)
  		to[out=222,in=222,looseness=\loos] (d14) to[out=203,in=203,looseness=\loos] (d15)
  		to[out=184,in=184,looseness=\loos] (d16) to[out=165,in=165,looseness=\loos] (d17)
  		to[out=146,in=146,looseness=\loos] (d18) to[out=127,in=127,looseness=\loos] (d19)
  		-- (c1)
  };
  \draw \foreach \x in {1,...,\nn} \foreach \y in {a,b,c,d} {
		(\y\x) node[circle, draw, fill=white, inner sep=0pt, minimum width=4pt] {} };
  \draw \foreach \x in {1,10} \foreach \y in {a,b,c,d} {
		(\y\x) node[circle, draw, fill=black, inner sep=0pt, minimum width=4pt] {} };
  \draw \foreach \x in {(c3),(a6),(d8),(d17)}{
		\x node[circle, draw, fill=gray!80, inner sep=0pt, minimum width=4pt] {} };
  \draw {
  		(57:\dc+0.5) node[] {$r$}
  		(355:\da-0.3) node[] {$x$}
  		(315:\dc+0.5) node[] {$u$}
  		(143:\dc+0.5) node[] {$y$}
  };
  \draw (140:\dc+1.5) node[] {$J_{19}:$};
  \draw (0:0.8) node[] {$I$};
  \draw (0:-0.5) node[] {$J$};
  \draw {
  		(90:\dc+0.65) node[] {$T_1$}
  		(280:\dc+0.7) node[] {$T_{l+1}$}
  };
 \end{scope}
 \end{tikzpicture}
\caption{An example of the last case of the proof of Theorem \ref{thm:FS1solid} where $n=21$ and $m=19$.}\label{fig:FSreduction}
\end{figure}
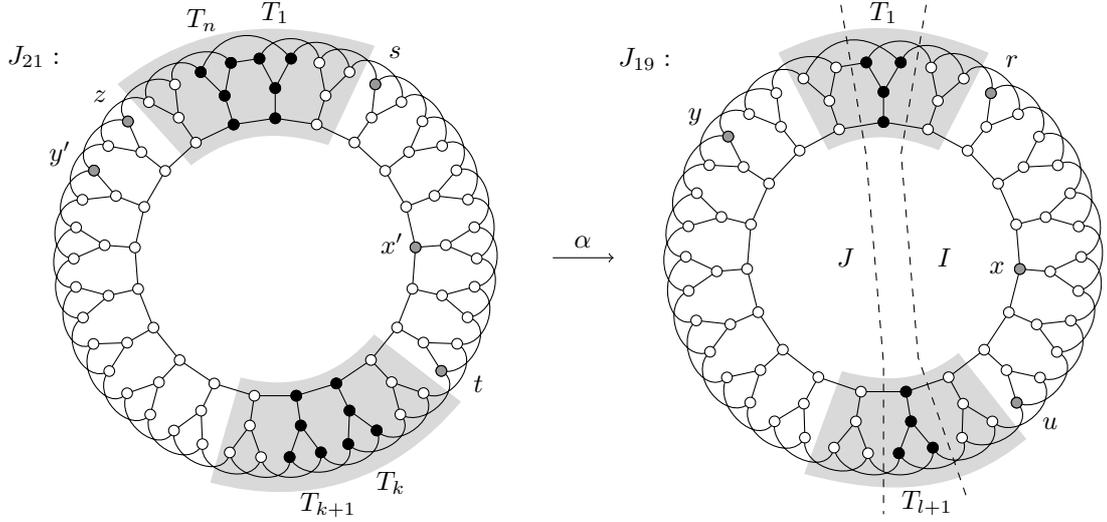

Let $R = \{ \alpha (s) \ | \ s \in S \}$. Since the stars $T_i^n$, where $i \in \{ 1,2,k-1,k,k+1,k+2,n-1,n \}$, do not contain elements of $S$, the vertices $\alpha(T_i^n)$ in $J_m$ are not in $R$, and $|R| = |S|$. In other words, the stars $T_j^m$, where $j \in \{ 1,2,l,l+1,l+2,m \}$, do not contain elements of $R$. We denote 
\begin{align*}
I = \bigcup\limits_{i=2}^{l} T_i^m \qquad \text{and} \qquad J = \bigcup\limits_{i=l+2}^{m} T_i^m
\end{align*}
(see Figure \ref{fig:FSreduction}). Let $r \in R$ and $v \in V(J_m)$. Let $s=\alpha^{-1} (r)$ and $v'$ be a preimage of $v$. Since the stars $T_i^m$ where $i \in \{ 1,2,l,l+1,l+2,m \}$ do not contain elements of $R$, the shortest paths from $r$ to $v$ are closely related to the shortest paths from $s$ to $v'$. 

Let us denote by $d_n$ and $d_m$ the distances in $J_n$ and $J_m$, respectively. If $r,v \in I$ or $r,v \in J$, then we have $d_m(r,v) = d_n(s,v')$. If $r \in I$ and $v \in J$, or $r \in J$ and $v \in I$, then $d_m(r,v) = d_n(s,v')-1$. If $v \in T_1^m$ or $v \in T_{l+1}^m$, then $d_m(r,v) = d_n(s,v')$ or $d_m(r,v) = d_n(s,v')-1$ depending on which of the preimages of $v$ the vertex $v'$ is. Indeed, let $\alpha^{-1} (v) = \{x_1, x_n\}$. If $r \in I$, then $d_m(r,v) = d_n(s,x_1) = d_n(s,x_n) - 1$. If $r \in J$, then $d_m(r,v) = d_n(s,x_n) = d_n(s,x_1) - 1$.


Let $x,y \in V(J_m)$ be distinct. In what follows, we will show that the set $R$ satisfies \eqref{eq:superb}. Suppose first that $x,y \in T_1^m$ or $x,y \in T_{l+1}^m$. Due to symmetry, it suffices to show that when $x,y \in T_1^m$ there exists an element $r \in R$ such that $d_m(r,x) < d_m(r,y)$. Let $x'$ and $y'$ be the preimages of $x$ and $y$ that are in $T_1^n$. Since $S$ is a 1-solid-resolving set of $J_n$, there exists some $s \in S$ such that $d_n(s,x') < d_n(s,y')$. Now, we have $d_m(\alpha (s), x) = d_n(s,x')$ if and only if $d_m(\alpha (s), y) = d_n(s,y')$. Consequently, $d_m(\alpha (s), x) < d_m(\alpha (s), y)$.


Suppose then that $x \in T_1^m$ and $y \in T_{l+1}^m$ (the case where $x \in T_{l+1}^m$ and $y \in T_1^m$ goes similarly). Assume to the contrary that there does not exist any $r \in R$ such that $d_m(r,x) < d_m(r,y)$. We have $d_m(r,x) \geq d_m(r,y)$ for all $r \in R$. Let $x_1$ and $x_n$ be the preimages of $x$ that are in the stars $T_1^n$ and $T_n^n$, respectively. Similarly, let $y_k$ and $y_{k+1}$ be the preimages of $y$ in the stars $T_k^n$ and $T_{k+1}^n$, respectively. Let $s \in S$. If $\alpha (s) \in I$, then we have $d_n(s,x_1) = d_m(\alpha (s),x) \geq d_m(\alpha (s),y) = d_n(s,y_k)$. Since $d_n(s, x_n) = d_n(s,x_1) + 1$ and $d_n(s,y_{k+1}) = d_n(s,y_k) + 1$, we have $d_n(s,x_n) \geq d_n(s,y_{k+1})$. If $\alpha (s) \in J$, then $d_n(s,x_n) = d_m(\alpha (s), x) \geq d_m(\alpha (s),y) = d_n(s,y_{k+1})$. Thus, we have $d_n(s,x_n) \geq d_n(s,y_{k+1})$ for all $s\in S$, a contradiction. Therefore, there must exist some $r \in R$ such that $d_m(r,x) < d_m(r,y)$. 


Suppose that $x \in T_1^m \cup T_{l+1}^m$ and $y \notin T_1^m \cup T_{l+1}^m$. Assume that $x \in T_1^m$ (the case where $x \in T_{l+1}^m$ follows by symmetry). We denote $y' = \alpha^{-1} (y)$ and $\alpha^{-1} (x) = \{x_1,x_n\}$, where $x_1 \in T_1^n$ and $x_n \in T_n^n$. Suppose that $y \in I$ (the case where $y \in J$ goes similarly). Assume to the contrary that $d_m(v,x) \geq d_m(v,y)$ for all $v \in R$. Let $r,u \in R$ be such that $r \in I$ and $u \in J$. We denote $s = \alpha^{-1} (r)$ and $t = \alpha^{-1} (u)$. Now we have $d_n(s,y') = d_m(r,y)$, $d_n(t,y') = d_m(u,y) + 1$, $d_m(r,x) = d_n(s,x_1)$ and $d_m(u,x) = d_n(t,x_1)-1$. Since $d_m(v,x) \geq d_m(v,y)$ for all $v \in R$, we have $d_n(s,x_1) = d_m(r,x) \geq d_m(r,y) = d_n(s,y')$ and $d_n(t,x_1) = d_m(u,x) + 1 \geq d_m(u,y) + 1 = d_n(t,y')$. Thus, for all $v' \in S$ we have $d_n(v',x_1) \geq d_n(v',y')$ and the set $S$ does not satisfy \eqref{eq:superb}, a contradiction. Thus, for some $v \in R$ we have $d_m(v,x) < d_m(v,y)$. Similarly, if $d_m(v,y) \geq d_m(v,x)$ for all $v \in R$, then $d_n(s,y') = d_m(r,y) \geq d_m(r,x) = d_n(s,x_1)$ and $d_n(t,y') = d_m(u,y) + 1 \geq d_m(u,x) + 1 = d_n(t,x_1)$. Consequently, for all $v' \in S$ we have $d_n(v',y') \geq d_n(v',x_1)$ and the set $S$ does not satisfy \eqref{eq:superb}, a contradiction. Thus, we also have $d_m(v,y) < d_m(v,x)$ for some $v \in R$.


Finally, assume that $x,y \notin T_1^m \cup T_{l+1}^m$. Let us denote $x' = \alpha^{-1} (x)$ and $y' = \alpha^{-1} (y)$. Assume that $x,y \in I$. Let $s \in S$ be such that $d_n(s,x') < d_n(s,y')$. Denote $r = \alpha (s)$. If $r \in I$, then $d_m(r,x) = d_n(s,x')$ and $d_m(r,y) = d_n(s,y')$. If $r \in J$, then $d_m(r,x) = d_n(s,x') - 1$ and $d_m(r,y) = d_n(s,y') - 1$. In both cases we have $d_m(r,x) < d_m(r,y)$. 
Thus, the set $R$ satisfies \eqref{eq:superb} for any $x,y \in I$. The case where $x,y \in J$ goes similarly.


Suppose that $x \in I$ and $y \in J$. There is at least one star between the stars that contain $x$ and $y$. We have the following two cases

\begin{enumerate}
\item There is exactly one star between $x$ and $y$:

Since $x \in I$ and $y \in J$, the star between $x$ and $y$ is either $T_1^m$ or $T_{l+1}^m$. Thus, there are two stars between $x'$ and $y'$. Suppose that $T_1^m$ is the star between $x$ and $y$, and $x \in T_2^m$ and $y \in T_m^m$. We have $x' \in T_2^n$ and $y' \in T_{n-1}^n$.
Let $x_1 \in T_1^n$ and $y_n \in T_n^n$ be such that they are the same type as $x'$ and $y'$, respectively. By 'same type' we mean that if $x'=c_2$, for example, then $x_1 = c_1$. 
Since the stars $T_i^n$ where $i \in \{ 1,2,k-1,k,k+1,k+2,n-1,n \}$ do not contain any elements of $S$, the vertex $s\in S$ is on the same side as $x'$ (that is, $\alpha (s) \in I$) if and only if we have $d_n(s,x') < d_n(s,y')$ Similarly, $s$ is on the same side as $y'$ if and only if $d_n(s,y') < d_n(s,x')$.

Assume that for all $v \in R$ we have $d_m(v,x) \geq d_m(v,y)$. Since $S$ is a 1-solid-resolving set of $J_n$, there exist vertices $s,t \in S$ such that $d_n(s,x') < d_n(s,y')$ and $d_n(t,y') < d_n(t,x')$. According to our previous observation, $s$ is on the same side as $x'$ and $t$ is on the same side as $y'$. Denote $r = \alpha (s)$. Since $d_m(r,x) \geq d_m(r,y)$, we have $d_n(s,y') - 1 \geq d_n(s,x') = d_m(r,x) \geq d_m(r,y) = d_n(s,y')-1$. Consequently, $d_m(r,x) = d_m(r,y)$ and $d_n(s,y') = d_n(s,x') + 1$. Thus, we have $d_n(s,x_1) = d_n(s,x') + 1 = d_n(s,y')$. Since the stars $T_i^n$, where $i \in \{ 1,2,k-1,k,k+1,k+2,n-1,n \}$, do not contain any elements of $S$, all shortest paths from $t$ to $x_1$ go through the star that contains $y'$. Since the star $T_n^n$ is between $T_1^n$ and the star that contains $y'$, we have $d_n(t,y') \leq d_n(t,x_1)$. Thus, the set $S$ does not satisfy \eqref{eq:superb} for $x_1$ and $y'$, a contradiction. Similarly, if $d_m(v,y) \geq d_m(v,x)$ for all $v \in R$, the set $S$ does not satisfy \eqref{eq:superb} for $y_n$ and $x'$. 

\item There are at least two stars between $x$ and $y$:

Now, there are at least three stars between $x'$ and $y'$. Let $s \in S$ be such that $d_n(s,x') < d_n(s,y')$, and denote $r = \alpha (s)$. As $d_m(r,y) \geq d_n(s,y') - 1$, we have $d_m(r,x) \leq d_m(r,y)$. If $d_m(r,x) < d_m(r,y)$, then we are done. Suppose that $d_m(r,x) = d_m(r,y)$. We have $r \in I$ since otherwise $d_m(r,x) = d_n(s,x') - 1 < d_n(s,y') - 1 = d_m(r,y)-1$. Since $d_m(r,x) = d_n(s,x')$ and $d_m(r,y) = d_n(s,y') - 1$, we have $d_n(s,x') = d_n(s,y') -1$. Clearly, there does not exist a shortest path from $r$ to $x$ that goes through the star that contains $y$. Since there are at least two stars between $x$ and $y$, there does not exist a shortest path from $r$ to $y$ that goes through the star that contains $x$. Indeed, otherwise we would have $d_m(r,x) \leq \rho_m(r,x) + 2 < \rho_m(r,y) \leq d_m(r,y)$. Thus, the shortest paths $r-x$ and $r-y$ can coincide with each other only in the star that contains $r$.

Let $z \in V(J_n)$ be the unique vertex that is the same type as $y'$ (i.e. $a_i$, $b_i$, $c_i$ or $d_i$), is in a star next to $y'$ and for which $d_n(s,z) = d_n(s,x')$ holds (see Figure \ref{fig:FSreduction}). The vertex $z$ is indeed unique since the first two conditions reduce the options to two and the third condition uniquely determines $z$ as $n$ in odd. Since $S$ is a 1-solid-resolving set of $J_n$, there exists a $t \in S$ such that $d_n(t,x') < d_n(t,z)$. If the vertex $t$ is in the same star as $y'$ or $z$, then $\rho_n(t,x') \geq 3$ since $\rho_n (y',x') \geq 4$. However, now $d_n(t,z) \leq 3 \leq \rho_n(t,x') \leq d_n(t,x')$. Thus, $t$ is not in the same star as $y'$ or $z$, and we have 
\begin{align*}
& d_n(t,y')-1 \leq d_n(t,z) \leq d_n(t,y') + 1, \\
& d_n(t,z)-1 \leq d_n(t,y') \leq d_n(t,z) + 1.
\end{align*}
If $d_n(t,y') < d_n(t,x')$, then $d_n(t,z) \leq d_n(t,y')+1 \leq d_n(t,x')$, a contradiction. Thus, we have $d_n(t,y') \geq d_n(t,x')$.

We denote $u = \alpha (t)$. If $u \in J$, then $d_m(u,y) = d_n(t,y')$ and $d_m(u,x) = d_n(t,x') -1$. Consequently, $d_m(u,x) \leq d_n(t,y') -1 < d_m(u,y)$ and \eqref{eq:superb} is satisfied for $x$ and $y$.

Suppose then that $u \in I$. Now, $d_m(u,y) = d_n(t,y') -1$ and $d_m(u,x) = d_n(t,x')$. If $d_m(u,x) < d_m(u,y)$, then \eqref{eq:superb} is again satisfied. Assume that $d_m(u,x) \geq d_m(u,y)$. Since there are at least two stars between $x$ and $y$, a shortest path $u-y$ cannot go through the star that contains $x$. Consequently, there is no shortest path $t-y'$ that goes through the star that contains $x'$. If there is a shortest path $t-y'$ that goes through the star that contains $z$, then we have $d_n(t,y') = d_n(t,z) + 1$ and
\begin{align*}
d_m(u,y) = d_n(t,y') - 1 = d_n(t,z) > d_n(t,x') = d_m(u,x).
\end{align*}
Thus, $u$ satisfies \eqref{eq:superb} for $x$ and $y$.
Suppose then that there is no shortest path $t-y'$ that goes through the star that contains $z$. The shortest paths $t-y'$ and $s-y'$ can coincide only in the star that contains $y'$. Consequently, the shortest paths $u-y$ and $r-y$ in $J_m$ can coincide only in the star that contains $y$. As we have seen before, the shortest paths $r-x$ and $r-y$ can coincide only in the star that contains $r$, and the shortest paths $u-y$ do not go through the star that contains $x$. Thus, the shortest paths $u-y$, $u-x$, $r-x$ and $r-y$ can coincide only in the stars that contain $x$, $y$, $r$ or $u$ (see Figure \ref{fig:FSreduction} for an example of this situation). Consequently, we have 
\begin{align*}
& d_m(r,x) + d_m(u,x) \leq \rho_m(r,x) + 2 + \rho_m(u,x) + 2 = \rho_m(r,u) + 4 \leq l-4 + 4 = l, \\
& d_m(r,y) + d_m(u,y) \geq \rho_m(r,y) + \rho_m(u,y) = m - \rho_m(r,u) \geq l+5.
\end{align*}
Thus, $d_m(r,x) + d_m(u,x) < d_m(r,y) + d_m(u,y)$. Since $d_m(r,x) = d_m(r,y)$, we have $d_m(u,x) < d_m(u,y)$.
Using similar arguments we can show that there exists some $u' \in R$ such that $d_m(u',y) < d_m(u',x)$.
\end{enumerate}
\end{proof}

\end{document}